\documentclass[draft]{amsart}

\usepackage{amssymb,amsmath,amsthm}
\usepackage{enumerate,color}

\usepackage[colorlinks=true, pdfstartview=FitV, linkcolor=blue, citecolor=blue, urlcolor=blue,pagebackref=false]{hyperref}

\usepackage{esint}
\usepackage{mathrsfs}

\usepackage{mathtools}

\newtheorem{theorem}{Theorem}
\newtheorem{lemma}[theorem]{Lemma}
\newtheorem{proposition}[theorem]{Proposition}

\theoremstyle{definition}
\newtheorem{definition}[theorem]{Definition}
\newtheorem{remark}[theorem]{Remark}

\numberwithin{equation}{section}

\newcommand{\R}{\mathbf{R}}

\DeclareMathOperator{\supp}{\operatorname{supp}}

\DeclarePairedDelimiter\ceil{\lceil}{\rceil}
\DeclarePairedDelimiter\floor{\lfloor}{\rfloor}

\begin{document}

\title[***]{Quantitative classification of potential Navier-Stokes singularities beyond the blow-up time }

\author[T. Barker]{Tobias Barker}
\address{Department of Mathematical Sciences, University of Bath, Bath BA2 7AY. UK}
\email{tb2130@bath.ac.uk}



\begin{abstract}
In \cite{hou}, Hou gave a compelling numerical candidate for a singular solution of the 3D Navier-Stokes equations. We pioneer classifications of potentially singular solutions, motivated by the issue of investigating the viability of numerical candidates.

For approximately axisymmetric initial data, we give the first quantitative classification of potentially singular solutions at \textit{any} given time in the region of potential blow-up times. Moreover, the quantitative bounds in the vicinity of any potential blow-up time are in principle amenable to numerical testing.

To achieve this, we establish improved quantitative regions of regularity for approximately axisymmetric initial data, which may be of independent interest. Together with improved quantitative energy estimates from \cite{TB24}, this allows us to get a quantitative lower bound in the vicinity of a blow-up time by implementing the strategy of \cite{BP21}, which is a physical space analogue of Tao's strategy \cite{Ta21} for producing quantitative estimates for critically bounded solutions. To obtain a quantitative lower bound on the solution at any time in the region of potential blow-up times, we recursively apply quantitative Carleman inequality arguments from \cite{Ta21}. This necessitates careful bookkeeping to avoid exponential losses and to ensure that all forward-in-time iterations of (localized) vorticity concentration remain within the region of quantitative regularity of the solution.

Finally, as an application of these methods we show that blow-up at $T^*$ at certain rates from the left imply blow-up from the right of $T^*$. To the best of our knowledge this is the first such result establishing sufficient conditions for blow-up from the right in terms of blow-up from the left.

\end{abstract}

\maketitle

\section{Introduction}
 The advancement in computational methods have played a key role in the search for singular solutions of the 3D Navier-Stokes equations
\begin{equation}\label{eq:NSE}
\begin{aligned}
	&v_t-\Delta v+v\cdot\nabla v+\nabla \pi=0
	&\textrm{div}\,v=0,\qquad v(\cdot,0)=v_0
\end{aligned}
\end{equation}

and Euler equations. For the Euler equations with a boundary, around a decade ago Luo and Hou \cite{houlou} gave a compelling scenario for a singularity forming from a smooth solution. This scenario was recently reexamined using neural networks by Wang, Lai, G\'{o}mez-Serrano and Buckmaster in \cite{buckmaster}, which was also used to investigate asymptotic self-similar blow-up profiles for other equations (see also further developments in \cite{buckmasterunstable}). Remarkably in \cite{chenhouI}-\cite{chenhouII} Chen and Hou gave a computer assisted proof of the Luo and Hou scenario.
For the axisymmetric 3D Navier-Stokes equations, in \cite{hou} Hou found a numerical candidate for singularity formation in a cylindrical domain with periodicity
\begin{equation}\label{eq:omegadef}
\Omega:=\{(r,z): 0\leq r\leq 1\quad\textrm{and}\quad z\in\mathbf{T}\}
\end{equation}
and Dirichlet boundary conditions imposed. Furthermore in \cite{hou}, Hou investigates the viability of this scenario by computing the growth of quantities involving various known regularity criteria for the 3D Navier-Stokes equations, such investigations motivate this paper.

The classification of solutions to the Navier-Stokes equations \eqref{eq:NSE} with potential singularities was pioneered by Leray in \cite{Le}.
Leray showed that if a weak Leray-Hopf solution\footnote{See Definition \ref{Lerays} for the definition of a `weak Leray-Hopf solution'.} ${v}:\R^3\times (0,\infty)\rightarrow\R^3$, with sufficiently smooth initial data, first loses smoothness at time $T^*>0$ (referred to as `the first blow-up time') then necessarily
\begin{equation}\label{Leraynecessary}
\|v(\cdot,t)\|_{L^{p}(\R^3)}\geq \frac{C(p)}{(T^*-t)^{\frac{1}{2}(1-\frac{3}{p})}}\quad\textrm{{for}}\,\,\textrm{{all}}\quad 0\leq t<T^*\quad\textrm{and}\quad 3< p\leq\infty.
\end{equation}
Furthermore in \cite{ladyzhenskaya}, \cite{prodi} and \cite{serrin}, it was shown by Ladyzhenskaya, Prodi and Serrin that if a weak Leray-Hopf solution $v:\R^3\times (0,\infty)\rightarrow\R^3$ satisfies
\begin{equation}\label{eq:LPS}
\int\limits_{0}^{T}\|v(\cdot,s)\|_{L^{p}(\R^3)}^{q} ds<\infty\quad\textrm{with}\quad \frac{3}{p}+\frac{2}{q}=1\quad\textrm{and}\,\,p>3
\end{equation}
then $v$ is smooth on $\R^3\times (0,T]$. In seminal work \cite{ESSARMA03}-\cite{ESS03}, Escauriaza, Seregin and \v{S}ver\'{a}k proved that if $v:\R^3\times (0,\infty)\rightarrow\R^3$ is a weak Leray-Hopf solution that first loses smoothness at $T^*>0$ then
\begin{equation}\label{eq:ESS}
\limsup_{t\uparrow T^*}\|v(\cdot,t)\|_{L^{3}(\R^3)}=\infty.
\end{equation} 
For further extensions see, for example, \cite{Se12}, \cite{Ph15}, \cite{GKP16} and \cite{Al18}. The result \eqref{eq:ESS} and these extensions are qualitative results. In \cite{Ta21}, Tao quantified \eqref{eq:ESS}, showing that
\begin{equation}\label{eq:L3quant}
\limsup_{t\uparrow T^*}\frac{\|v(\cdot,t)\|_{L^{3}(\R^3)}}{(\log\log\log(\frac{1}{T^*-t}))^c}=\infty
\end{equation}
for some positive universal constant $c$. Subsequent quantitative results in this direction were then obtained in \cite{BP21}, \cite{Pa21}, \cite{TB23} and \cite{quantiativebesov}, for example.

Though it is known\footnote{\label{nosingaxis}For \eqref{eq:L3quant}, this follows from \cite{TB23}, together with the fact that for suitable weak Leray-Hopf axisymmetric solutions with Dirichlet boundary condition on $r=1$, singularities can only occur on $r=0$. This last fact readily follows from \cite{CKN} and  \cite{sereginboundaryreg}.} that the regularity criteria \eqref{eq:LPS}-\eqref{eq:L3quant} also hold for axisymmetric solutions to \eqref{eq:NSE} on $\Omega$ \eqref{eq:omegadef} (with $\R^3$ replaced by $\Omega$), one may encounter several issues when examining the viability of numerical candidates for singular solutions through the lens of such regularity criteria/necessary conditions for blow up. The use of regularity criteria over the whole domain where the solution is defined can create difficulties when examining such quantities numerically. For instance  in \cite[pg 2288]{hou} it is stated that  `Our current adaptive mesh
strategy only provides sufficient resolution in the near field, but the adaptive grid in
the far field is relatively coarse.`. It is also mentioned in \cite[pg 2288]{hou} that $\|v(\cdot,t)\|_{L^{3}(\Omega)}$ appears to decrease at certain points in time due to this. Another issue is that the quantitative regularity criteria \eqref{eq:L3quant} involves triple logarithmic growth. Growth at such a slow rate seems not possible to observe numerically: `it would be almost impossible to capture such slow growth rate
numerically with our current computational capacity' \cite[pg 2286]{hou}.

One of the most serious issues is that all currently known necessary conditions for blow-up  (that are not in terms of the initial data) involve a formulation in terms of the blow-up time $T^*$ or information about some quantity of the solution over a time interval. Such aspects make it practically impossible to use such necessary conditions to test the viability of numerically computed candidates of singular solutions. For example, Hou's numerically computed axisymmetric solution ${v}$ of \eqref{eq:NSE}, which is a candidate for a potential singular solution, seems to satisfy the bound
\begin{equation}\label{vtypeI}
\|{v}(\cdot,t)\|_{L^{\infty}(\Omega)}\sim \frac{1}{(T^*-t)^{\frac{1}{2}}}\quad\textrm{{for}}\quad 0<t<T^* ,
\end{equation} 

{where $T^*$ is the first blow-up time. However, it is known that axisymmetric solutions satisfying \eqref{vtypeI} cannot lose smoothness at $T^*$ \cite{SSaxisymmetric}\footnote{This  also requires that axisymmetric solutions with Dirichlet boundary condition at $r=1$ cannot possess singularities away from the axis. See the references in footnote \ref{nosingaxis}. For earlier results ruling at Type I blow-up for axisymmetric solutions defined on $\R^3$ see, for example, \cite{chenaxisymmetricI}, \cite{chenaxisymmetricII} and \cite{KNSS}.}. It is stated by Hou in \cite[pg 2278]{hou} that inequalities such as \eqref{vtypeI} `would be almost impossible to verify numerically since it requires an exact value of $T^*$'.}

This paper seeks to address these issues by pioneering quantitative characterizations of  potentially singular solutions of \eqref{eq:NSE}, with approximately axisymmetric initial data, that
\begin{itemize}

\item give numerically viable quantitative lower bounds in the vicinity of any blow-up time,
\item involves quantities of the velocity field over compact spatial domains and
\item  apply at any single moment in time in the temporal region of irregularity and do not refer to the blow-up time.
\end{itemize}
As a direct application of the methods used to prove Theorems \ref{thm:loweratblowup}-\ref{thm:lowernoref}, we show that divergence of scale-invariant norms near $T^*$ at a certain rate necessarily implies blow-up from the right of $T^{*}$. To the best of our knowledge, this is the first result establishing sufficient conditions for blow-up from the right of a blow-up time to occur.
\subsection{Main results}
In the theorems below and throughout this paper, we use the notation
\begin{equation}\label{eq:lambdaolambda1def}
\lambda_{0}:=\frac{1}{128000000},\quad \lambda_{1}:=
\frac{\lambda_{0}}{2(1+\lambda_0)}.
\end{equation}
For the definitions of `approximately axisymmetric', `suitable weak Leray-Hopf solution' and `singular point', we refer the reader to Definition \ref{approxaxi} and the subsection `Notions of solutions and singularity'.
\begin{theorem}\label{thm:loweratblowup}
Let $p\in (3,\infty]$. There exists a $0<c_{0}(p)<c_{1}(p)$, $c_{2}(p)$ and a  positive universal constant $M_0>1$ such that the following statement holds true.\\
Let $v:\R^3\times (0,\infty)\rightarrow \R^3$ be a suitable weak Leray-Hopf solution, with initial data $v_{0}:\R^3\rightarrow\R^3$.
Suppose that $v_{0}(x_1,x_2,x_3)$ is approximately axisymmetric and there exists $M\geq M_{0}$ such that
\begin{equation}\label{initialdataassumption1}
\|v_{0}\|_{L^{3}(\R^3)}+\||x_3|^{1-\frac{3}{p}}|v_{0}(\cdot,x_3)|\|_{L^{p}(\R^3)}\leq M.
\end{equation}
Moreover, suppose that $v$ possesses a singular point at $T^*$.\\
Then the above assumptions imply that
\begin{equation}\label{eq:vloweratblowup}
\int\limits_{M^{c_0(p)} T^{\frac{1}{2}}<|y|<M^{c_1(p)} T^{\frac{1}{2}}} |v(y,T)| dy\geq T e^{-M^{c_2(p)}}\qquad\forall T\in [T^*, T^*(1+\lambda_0)].
\end{equation}
\end{theorem}
\begin{theorem}\label{thm:lowernoref}
Let $p\in (3,\infty]$. There exists a $0<c_{3}(p)<c_{4}(p)$, $c_{5}(p)$ and a positive universal constant $M_1>1$ such that the following statement holds true.\\
Let $v:\R^3\times (0,\infty)\rightarrow \R^3$ be a suitable weak Leray-Hopf solution, with initial data $v_{0}:\R^3\rightarrow\R^3$.
Suppose that $v_{0}(x_1,x_2,x_3)$ is approximately axisymmetric and there exists $M\geq M_{1}$ such that
\begin{equation}\label{initialdataassumption2}
\|v_{0}\|_{H^{1}(\R^3)}+\|v_0\|_{L^{3}(\R^3)}+\||x_3|^{1-\frac{3}{p}}|v_{0}(\cdot,x_3)|\|_{L^{p}(\R^3)}\leq M.
\end{equation}
Moreover, suppose that $v$ is not smooth on $\R^3\times (0,\infty)$.\\
Then the above assumptions imply that
\begin{equation}\label{eq:vlowernoref}
\int\limits_{M^{c_3(p)} <|y|<e^{M^{c_4(p)}}} |v(y,T)| dy\geq  e^{-e^{M^{c_5(p)}}}\qquad\forall T\in [M^{-5},M^{5}].
\end{equation}
Furthermore, if $v$ possesses a singular point at $T^*$ then
\begin{equation}\label{eq:vlowerbeforeblowup}
\int\limits_{M^{c_0(p)} T^{\frac{1}{2}}<|y|<M^{c_1(p)} T^{\frac{1}{2}}} |v(y,T)| dy\geq T e^{-M^{c_2(p)}}\qquad\forall T\in [M^{-5}, T^*(1+\lambda_0)].
\end{equation}
Here, $c_0(p)-c_{2}(p)$ and $\lambda_0$ are as in Theorem \ref{thm:loweratblowup}.
\end{theorem}
Lastly as a direct application of these methods, we show that for certain types of potential Navier-Stokes singularities occurring at $t=T^*$, blow-up necessarily occurs from the right of $T^*$ as well. 
\begin{theorem}\label{thm:blowupright}
Let $q\in (3,\infty]$, $0<c<\frac{1}{2}(1-\frac{3}{q})$ and $T^*>0$. There exists $0<\varepsilon_{0}(c, T^*)<T^*$, $0<\varepsilon_{1}(c,T^*)$, $0<\varepsilon_{2}(T^*)<T^*$ and $0<\varepsilon_{3}(T^*)$ such that the following statements hold true.\\
Let $v:\R^3\times (0,\infty)\rightarrow \R^3$ be a suitable weak Leray-Hopf solution, with smooth compactly supported initial data $v_{0}:\R^3\rightarrow\R^3$.
Moreover, suppose that $v$  possesses a singular point at $(x,t)=(0,T^*)$ with $T^*>0$ and is smooth on $\R^3\times (0,T^*)$. \\
If
\begin{equation}\label{eq:logadjusted}
|v(x,t)|\leq \frac{\Big(\log\Big(\frac{1}{(T^*-t)^{c}}\Big)\Big)^{\frac{1}{1874}}}{|x|}\quad\textrm{for all}\,\,x\in\R^3\,\,\textrm{and}\quad T^*-\varepsilon_{0}(c, T^*)<t<T^*,
\end{equation}
then the above assumptions imply that
\begin{equation}\label{eq:Lpfromright}
\Big(\int\limits_{\R^3}|v(x,s)|^q dx\Big)^{\frac{1}{q}}\gtrsim_{c,q} \frac{1}{(s-T^*)^{\frac{1}{2}(1-\frac{3}{q})-c}}\quad\textrm{for all}\quad T^*<s<T^*+\varepsilon_{1}(c, T^*).
\end{equation}
If we instead assume
\begin{equation}\label{eq:doublelogadjusted}
|v(x,t)|\leq \frac{\Big(\log\log\Big(\frac{1}{(T^*-t)}\Big)\Big)^{\frac{1}{1876}}}{|x|}\quad\textrm{for all}\,\,x\in\R^3\,\,\textrm{and}\quad T^*-\varepsilon_{2}(T^*)\leq t<T^*,
\end{equation}
then we conclude that
\begin{equation}\label{eq:L3fromright}
\Big(\int\limits_{\R^3}|v(x,s)|^3 dx\Big)^{\frac{1}{3}}\geq \Big(\log\Big(\frac{1}{s-T^*}\Big)\Big)^{\frac{1}{2}}\quad\textrm{for all}\quad T^*<s<T^*+\varepsilon_{3}(T^*).
\end{equation}
\end{theorem}
\subsection{Comparison with previous literature and strategy}
\subsubsection{Obtaining exponentially small lower bounds of the solution at a blow-up time}
For a weak Leray-Hopf solution $v:\R^3\times (0,\infty)\rightarrow\R^3$ that first loses smoothness at $T^*>0$, information on $v(\cdot,T^*)$ can be inferred from Escauriaza, Seregin and \v{S}ver\'{a}k's seminal works \cite{ESSARMA03}-\cite{ESS03}. In particular, \cite[Theorem 15.4]{PGLR} demonstrates that arguments from  \cite{ESSARMA03}-\cite{ESS03}, based on backward uniqueness and unique continuation for parabolic differential inequalities, imply that $v(\cdot,T^*)$ cannot be identically equal to zero. Subsequently, in \cite[Theorem 4.1/Remark 4.2]{AB19} Albritton and the author extended such arguments to show that $v(\cdot, T^*)$ \textit{cannot be relatively small compared to the initial data $v_0$.} Namely, for all $p>3$ and $M>0$, there exists $\varepsilon(M,p)>0$ such that
\begin{equation}\label{eq:albrittonTB}
\|v_0\|_{\dot{B}^{-1+\frac{3}{p}}_{p,\infty}(\R^3)}\leq M\qquad\Rightarrow\qquad \|v(\cdot,T^*)\|_{\dot{B}^{-1}_{\infty,\infty}(\R^3)}\geq \varepsilon(M,p).
\end{equation}
The result \eqref{eq:albrittonTB} from \cite{AB19} is based on a contradiction argument and gives no quantitative information on $\varepsilon(M,p)$. The first quantitative lower bound on $v(\cdot,T^*)$ was given by Prange and the author in \cite[Proposition 4.4]{BP21}. There it was shown that if $(x,t)=(0,T^*)$ is a singular point of a suitable weak Leray-Hopf solution $v$ then for $M$ sufficiently large
\begin{equation}\label{eq:BP21lower}
\begin{split}
&\|v_0\|_{L^{3}(\R^3)}\leq M\quad\Rightarrow\\
&\|v(\cdot,T^*)\|_{L^{3}(B(0,\exp\exp(M^{C_{univ}})(T^*)^{\frac{1}{2}})\setminus B(0,(T^*)^{\frac{1}{2}}))}\geq \exp(-\exp(\exp(M^{C_{univ}}))).
\end{split}
\end{equation}
The presence of the triple exponential is well beyond being detectable numerically. Roughly speaking, each of the following contribute an exponential towards the triple exponentially small lower bound \eqref{eq:BP21lower}:
\begin{itemize}
\item[] 1. Use of the energy estimate\footnote{Here the $L^{\infty}_{t}L^{2}_{x}\cap L^{2}_{t}\dot{H}^{1}_{x}$ norm is defined by \eqref{energynormdef} and $e^{s\Delta}$ denotes the heat semi-group.} (with $\|v_0\|_{L^{3}(\R^3)}\leq M$):
\begin{equation}\label{eq:energyestgronwall}
\|v(\cdot,s)-e^{s\Delta }v_0\|_{L^{\infty}_{t}L^{2}_{x}\cap L^{2}_{t}\dot{H}^{1}_{x}(\R^3\times (0,t))}^2\leq \exp(M^{C_{univ}}) t^{\frac{1}{2}}.
\end{equation}
\item[] 2. Pigeonhole arguments involving the estimate \eqref{eq:energyestgronwall} to locate annuli of regularity of the form
\begin{equation}\label{eq:annuliintro}
\begin{split}
&\{x: |x|\in [RT^{\frac{1}{2}}, RM^{C_{univ}}T^{\frac{1}{2}}]\times (0,T)\quad\textrm{with}\\
&R\in[1,\exp\exp(M^{C_{univ}})],
\end{split}
\end{equation}
in which $v$ is quantitatively bounded. Such regions are used to implement quantitative Carleman inequalities, see 3. below.
\item[] 3. The use of arguments from Tao's paper \cite{Ta21} involving the use of quantitative backward uniqueness and unique continuation Carleman inequalities.
\end{itemize}
Our first observation to improve the lower bound of $v(\cdot,T^*)$ compared to \eqref{eq:BP21lower} in Theorem \ref{thm:loweratblowup} is to use the improved energy estimate
\begin{equation}\label{eq:improvedenergyintro}
\|v(\cdot,s)-e^{s\Delta }v_0\|_{L^{\infty}_{t}L^{2}_{x}\cap L^{2}_{t}\dot{H}^{1}_{x}(\R^3\times (0,t)}^2\lesssim M^{12} t^{\frac{1}{2}}
\end{equation}
from \cite[Lemma 1]{TB24}. Implementing this in the proof in \cite[Proposition 4.4]{BP21} would result in a double exponential lower bound in \eqref{eq:BP21lower}.

 To remove a further exponential in the lower bound for $v(\cdot,T^*)$ in Theorem \ref{thm:loweratblowup}, the second observation we use is the following. Namely, that suitable weak Leray-Hopf solutions with approximately axisymmetric initial data\footnote{See Definition \ref{approxaxi}.} satisfying \eqref{initialdataassumption1} are quantitatively bounded in
\begin{equation}\label{eq:improvedannulusintro}
\{x: |x|\geq M^{700+\frac{30}{1-\frac{3}{p}}}T^{\frac{1}{2}}\}\times (0,T).
\end{equation}
This is given in  Proposition \ref{prop:improvedannulus} and may be of independent interest.  A related observation was previously obtained by Palasek in \cite{Pa21} for approximately axisymmetric solutions satisfying the scale invariant bound $\|v\|_{L^{\infty}_{t}L^{3}_{x}(\R^3\times (0,T))}\leq M$, with the region of regularity being of the form
\begin{equation}\label{eq:Palasekannuli}
\{x=(x_1,x_2,x_3): x_1^2+x_2^2\geq M^{C_{univ}} T\}\times (\tfrac{T}{2}, T).
\end{equation}
The scale-invariant bound $\|v\|_{L^{\infty}_{t}L^{3}_{x}(\R^3\times (0,T))}\leq M$ is more stringent than the setting of Theorems \ref{thm:loweratblowup}-\ref{thm:lowernoref}, where we only have a quantitative bound on the initial data. Moreover in \cite{Pa21}, the solution is approximately axisymmetric, whereas in our setting only the initial data is. Having such symmetry assumptions only on the initial data is relevant when considering behavior beyond the blow-up time in Theorems \ref{thm:loweratblowup}-\ref{thm:lowernoref}, where solutions may potentially break symmetry classes such as being axisymmetric.

To obtain the region of regularity \eqref{eq:improvedannulusintro}, we use local-in-space smoothing results from Kang, Miura and Tsai's paper \cite{KMT21}, together with the idea from \cite{Pa21} that concentrations for approximately axisymmetric functions away from the axis correspond to concentrations on a torus. For previous results using $\varepsilon-$regularity/local-in-space smoothing to establish regions of regularity and decay for the Navier-Stokes equations in terms of quantities associated to the initial data, see \cite[Theorems C-D]{CKN}, \cite{KMT21} and \cite[Theorem 14.6]{PGLR}, for example.
\subsubsection{Obtaining lower bounds of the solution beyond the blow-up time}
For a weak Leray-Hopf solution to the 3D Navier-Stokes, with initial data satisfying
\begin{equation}\label{eq:idH1intro}
\|v_0\|_{{H}^1(\R^3)}\leq M,
\end{equation}
the region containing potential blow-up times is known to be
\begin{equation}\label{eq:blowupset}
\Sigma_{irregular}:=\{T: C_{univ}^{-1}M^{-4}\leq T\leq C_{univ}M^4\}.
\end{equation}
Very little is known about (potentially) singular weak Leray-Hopf solutions at \textit{any} given time in the region $\Sigma_{irregular}$ given by \eqref{eq:blowupset}. It is expected that singular weak Leray-Hopf solutions will become non-unique after the first blow-up time $T^*>0$, which creates difficulties in classifying their behavior beyond the blow-up time\footnote{For results on separation rates of potentially non-unique solutions, see \cite[Remarks 1.3]{BSS18}, \cite{phelphs23}, \cite{bradshawphelps23}, \cite{bradshaw24} and \cite{bradshawhudson25}, for example}.

Despite no recorded results, existing arguments in the literature give that there exists $\varepsilon(M)>0$ such that if a suitable weak Leray-Hopf solution $v:\R^3\times (0,\infty)\rightarrow\R^3$, with initial data satisfying \eqref{eq:idH1intro}, loses smoothness then
\begin{equation}\label{eq:vlowersoft}
\|v(\cdot,T)\|_{L^2(\R^3)}\geq \varepsilon(M)\quad\forall T\in\Sigma_{irregular}.
\end{equation}
This is obtained by soft arguments. If there does not exist  $\varepsilon(M)>0$ such that \eqref{eq:vlowersoft} holds, then one may argue in a similar way to \cite[First step-second step p. 66]{BP21mild} to show this implies the existence of a suitable weak Leray-Hopf solution $\bar{v}:\R^3\times (0,\infty)\rightarrow\R^3$ and initial data $\bar{v}_{0}$ such that
\begin{itemize}
\item[] 1. $\|\bar{v}_{0}\|_{H^1(\R^3)}\leq M,$
\item[] 2. $\bar{v}(\cdot, \bar{T})=0$ for some $\bar{T}\in \Sigma_{irregular}$,
\item[] 3. $\bar{v}$ loses smoothness.
\end{itemize}
From 1.-2. and arguments discussed in the previous subsection, we must have that $\bar{v}(x,t)\equiv 0$ in $\R^3\times (0,\infty)$, which contradicts 3.

The result \eqref{eq:vlowersoft} is purely qualitative, as it is achieved by soft contradiction arguments. In Theorem \ref{thm:lowernoref}, we seek a quantitative classification of a potentially singular solution $v$ in the region $\Sigma_{irregular}$ given by \eqref{eq:blowupset}. Starting from Tao in \cite{Ta21} and followed by \cite{BP21}, \cite{Pa21}, \cite{TB23}, \cite{quantiativebesov}, quantitative estimates were produced for smooth solutions to the 3D Navier-Stokes equations with bounded critical norms\footnote{These are norms for $v$ of the form $L^{\infty}_{t}X_x$, with the $X_x$ norm being preserved by the rescaling $\lambda v_{0}(\lambda y)$ with $\lambda>0$.}. In these papers, quantitative backward uniqueness and unique continuation are used to transfer space-time concentrations of the $L^2$ norm of the vorticity to a lower bound on the $L^2$ norm of the vorticity at a future time.
 One might hope that the same arguments could be utilized to transfer a lower bound on the vorticity in the vicinity of the first blow-up time $T^*$ to \textit{any} time in $\Sigma_{irregular}$. Formally, the difficulty in doing so is that the quantitative unique continuation used there for $w:B(0,r)\times [-T_1,0]\rightarrow\R^3$, satisfying certain differential inequalities, takes the form
\begin{equation}\label{eq:uniquecontintro}
\begin{split}
\textrm{L.H.S}=\int\limits_{-2\overline s}^{-\overline s}\int\limits_{|x|\leq \frac r2}(T_{1}^{-1}|w|^2+|\nabla w|^2)e^{\frac{|x|^2}{4t}}\, dxdt\quad\textrm{with}\quad 0<\overline{s}<\frac{T_{1}}{30000}\quad\textrm{and}\,\,\textrm{the}\\
\textrm{R.H.S}\quad\textrm{contains}\quad \Big(\frac{\overline s}{\underline s}\Big)^\frac{3}{2}\Big(\frac{3e\overline s}{\underline s}\Big)^{\frac{r^2}{200\overline s}}\int\limits_{|x|\leq r}|w(x,0)|^2 e^{-\frac{|x|^2}{4\underline s}}\, dx
\end{split}
\end{equation}
Due to the restriction on $\bar{s}$ in the above (with respect to $T_{1}$), one observes that when applying this form of quantitative unique continuation, only concentrations of $w$ that are close enough to the final moment of time can be transferred to give a concentration (lower bound) of $w$ at the final moment of time. This results in a block for using this quantitative Carleman inequality to transfer concentrations of vorticity near the first blow-up time $T^*$ to a lower bound on the (localized) $L^2$ norm of the vorticity at times far from $T^*$. A similar issue appears in \cite{ESSARMA03}, when trying to prove backward uniqueness for a function $u:\R^n\setminus B(0,R)\times [0,T]\rightarrow\R$, satisfying certain differential inequalities and growth bounds, with $u(\cdot,0)=0$. Such a comparison is not coincidental, as the Carleman inequalities used in \cite{Ta21} and subsequent quantitative Navier-Stokes works are closely related to those in \cite{ESSARMA03}.

In order to overcome this difficulty, we establish a key inductive step showing that the quantitative Carleman inequalities from Tao's work \cite{Ta21} can be used to propagate localized $L^2$ space-time vorticity concentration (lower bound) forward-in-time in a uniform way. In particular, in this inductive step it is crucial that the ratio of the future time interval (where the propagation of the concentration occurs) and the past time interval (where the concentration originates) is fixed. This enables the inductive step to be carefully iterated and for a lower bound on the (localized) $L^2$ norm of the vorticity to be obtained far beyond the first blow-up time. Related to this, let us mention that in the aforementioned work \cite{ESSARMA03}, to prove backward uniqueness of $u:\R^n\setminus B(0,R)\times [0,T]\rightarrow \R$, it is first shown that $u\equiv 0$ in $\R^n\setminus B(0,R)\times [0,\varepsilon(n)]$. This is then iterated to show backward uniqueness on $\R^n\setminus B(0,R)\times [0,T]$. The iteration procedure in this paper is, in some sense, a quantitative analogue of this.

Let us mention that in Tao's paper \cite{Ta21} that proves quantitative estimates for a smooth solution to the Navier-Stokes equations satisfying $\|v\|_{L^{\infty}_{t}L^{3}_{x}([-T,0]\times \R^3)}\leq M$, for certain $0<T_1\leq M^{-c^{(0)}_{univ}}T$ vorticity concentration of the form
$$\int\limits_{-T_1}^{-M^{-c^{(1)}_{univ}}T_1}\int\limits_{\frac{R}{2}\leq |x|\leq R} |\nabla\times v|^2 dxdt\geq \textrm{lower bound} $$
is transferred to a lower bound on $$\int\limits_{D} |\nabla\times v(x, 0)|^2 dx$$ (for a certain annulus $D\subset \R^3$) via quantitative Carleman inequalities. Notably, the vorticity concentration is transferred very slightly to the future, with the gap between the original time interval (where the concentration occurs) and the final moment in time (where the forward propagation of the concentration occurs) degenerating with respect to $M$. In the context of Theorem \ref{thm:lowernoref}, iterating such a degenerating gap in the past and future times would create an undesirable additional (triple) exponential in the lower bound. By means of careful bookkeeping, we show that such a degeneracy can be avoided, leading to a better double exponential lower bound for times in $\Sigma_{irregular}$. Bookkeeping is also necessary to track how the spatial domain containing the forward-in-time concentration changes with respect to the  forward-in-time iterations. This is required to ensure that each of these domains in the forward-in-time iteration is contained in the region \eqref{eq:improvedannulusintro} where $v$ is quantitatively bounded, which is a requirement for the application of the quantitative Carleman inequalities.
\subsection{Remarks}
\subsubsection{Initial data that is not approximately axisymmetric}
Though we do not pursue it in this paper, the interested reader can verify that the conclusions of Theorems \ref{thm:loweratblowup}-\ref{thm:lowernoref} also hold true for initial data  that are not necessarily approximately axisymmetric. In particular, they hold for initial data satisfying 
\begin{equation}\label{eq:nonaxisymiddecayTheorems}
\begin{split}
&\|v_0\|_{L^{3}(\R^3)}+\||x|^{1-\frac{3}{p}}|v_{0}(x)|\|_{L^{p}(\R^3)}\leq M\quad\textrm{and}\\
&\|v_0\|_{H^1(\R^3)}+\|v_0\|_{L^{3}(\R^3)}+\||x|^{1-\frac{3}{p}}|v_{0}(x)|\|_{L^{p}(\R^3)}\leq M\quad\textrm{for some}\quad p\in (3,\infty]
\end{split}
\end{equation}
respectively.
\subsubsection{Optimality}
We do not know if the lower bounds in Theorems \ref{thm:loweratblowup}-\ref{thm:lowernoref} are optimal or not. Note that if $v_0$ is divergence-free, compactly supported in $B(0,1)$ and satisfies
$$\|v_0\|_{H^1(\R^3)}\leq M,$$ 
then  for $M$ sufficiently large (independent of $p$) we have that for $M^{-5}\leq T\leq M^5$ that
$$\int\limits_{M^{708+\frac{30}{1-\frac{3}{p}}}T^{\frac{1}{2}}\leq |x|} |e^{T\Delta} v_0(x)| dx\leq Te^{-M^{1400}}. $$
Compared to the conclusion of Theorem \ref{thm:loweratblowup}, specifically the lower bound given by \eqref{eq:L2lowerproof}
\begin{equation*}
\begin{split}
\int\limits_{M^{708+\frac{30}{1-\frac{3}{p}}}T^{\frac{1}{2}}\leq |x|\leq M^{1102+\frac{30}{1-\frac{3}{p}}}T^{\frac{1}{2}}} &|v(x,T)| dx\geq Te^{-M^{1813+\frac{60}{1-\frac{3}{p}}}}\\
&\forall T\in [T^{*},T^{*}(1+\lambda_{0})],
\end{split}
\end{equation*}
 we see that the corresponding upper bound for the above linearized problem is of comparable (exponentially small) order. 
\subsubsection{Vertical decay on the initial data}
In \cite{Pa21}, it is claimed that for an axisymmetric  solution $v$ that first loses smoothness at $T^*>0$ one has
$$
 \limsup_{t\uparrow T^*}\frac{\|v(\cdot,t)\|_{L^{3}(\R^3)}}{(\log\log(\frac{1}{T^*-t}))^c}=\infty,
 $$
which improves \eqref{eq:L3quant}. An observation used for this in \cite{Pa21} is that an axisymmetric solution with $\|v\|_{L^{\infty}_{t}L^{3}_{x}(\R^3\times (0,T))}\leq M$, is quantitatively bounded in a region of the form
\begin{equation}\label{eq:Palasekannulirecall}
\{x=(x_1,x_2,x_3): x_1^2+x_2^2\geq M^{C_{univ}} T\}\times (\tfrac{T}{2}, T).
\end{equation}
One may therefore wonder if the vertical decay imposed on the initial data \eqref{initialdataassumption1} in Theorems \ref{thm:loweratblowup}-\ref{thm:lowernoref} is artificial or not. For establishing the quantitative estimates in \cite{Pa21}, there appears to be a gap in the proof caused by the anisiotropy of the region of regularity \eqref{eq:Palasekannulirecall}. Though in \cite{Pa21} an anisiotropic quantitative backward uniqueness Carleman inequality is developed, it is used in combination with quantitative unique continuation from \cite{Ta21}. This is isotropic (radial weight) and this combination seems to create a gap in the proof, related to regions of concentration resulting from pigeonholing remaining in the region of regularity \eqref{eq:Palasekannulirecall}. Instead, the vertical decay imposed on the initial data \eqref{initialdataassumption1} in Theorems \ref{thm:loweratblowup}-\ref{thm:lowernoref} gives an isotropic region of regularity  \eqref{eq:improvedannulusintro} and we do not encounter such issues.
\subsubsection{Blow-up from the right for Type I singularities.}
In the setting of Theorem \ref{thm:blowupright}, we see that if $v$ possesses a singular point at $(0,T^*)$ and satisfies the scale invariant bound (Type I singularity)
\begin{equation}\label{eq:TypeI}
|v(x,t)|\leq \frac{M}{|x|}\quad\textrm{for all}\quad 0<t<T^*
\end{equation}
for $M$ sufficiently large, then necessarily
\begin{equation}\label{eq:LpfromrightTypeI}
\begin{split}
&\Big(\int\limits_{M^{738}(s-T^*)^{\frac{1}{2}}\leq |x|\leq M^{1134}(s-T^*)^{\frac{1}{2}}}|v(x,s)|^q dx\Big)^{\frac{1}{q}}\geq \frac{e^{-M^{1875}}}{(s-T^*)^{\frac{1}{2}(1-\frac{3}{q})}}\\
&\textrm{for}\quad s\in (T^*,T^*(1+\lambda_{0}))\quad\textrm{and}\quad q\in [3,\infty].
\end{split}
\end{equation}
This can be achieved by essentially the same arguments as in the proof of Theorem \ref{thm:blowupright} (Case 1). This demonstrates that for Type I singularities \eqref{eq:TypeI}, the shrinker type concentration phenomenon from the left of $T^*$ in \cite[Theorem 2]{BP20} symmetrically occurs to the right of $T^*$ (expander behavior).
\subsubsection{A note on generalizations of Theorem \ref{thm:blowupright}}
Theorem \ref{thm:blowupright} has been stated in the simplest possible setting for the reader's convenience.
The same arguments show that the assumption that $v$ is smooth on $\R^3\times (0,T^*)$ (`first blow-up time') is unnecessary. In this case, the conclusions \eqref{eq:Lpfromright} and \eqref{eq:L3fromright} instead will apply on a certain sequence of times tending to $T^{*}$ from the right, determined by the location of times before $T^*$ where the solution is continuous in $L^{2}(\R^3)$ from the right.

Other than examining blow-up from the right in terms of quantitative behavior of the scaling invariant type quantity $|x||v(x,t)|$, we could readily consider other endpoint scale-invariant norms such as $L^{\infty}_{t}L^{3,\infty}_{x}$ or $(T^*-t)^{\frac{1}{5}}\|v(\cdot,t)\|_{L^{5}(\R^3)}$ with certain adjustments to the quantitative assumptions.  Note that for these methods, the spatial norm is required to decay. In particular, we do not know if blow-up from the right of $T^*$ occurs if the $L^{\infty}$ norm blows-up from the left of $T^*$ at a self-similar rate:
\begin{equation}\label{TypeIremark}
\|v(\cdot,t)\|_{L^{\infty}(\R^3)}\sim\frac{1}{(T^*-t)^{\frac{1}{2}}}.
\end{equation}
In the remarkable work \cite{CDP25}, solutions were constructed that are smooth on $\mathbf{T}^3\times (0,\infty)\setminus\{T^*\}$ and blow-up on either side of $T^*$ at the (sequential) rate \eqref{TypeIremark}. These solutions seem not to belong to the Leray-Hopf class in Definition \ref{Lerays}.

\subsection{Notation}
\subsubsection{Universal constants, vectors and domains}
When a constant depends on a parameter (for example $p$), we will sometimes denote its dependence by $c(p)$. In this paper, universal constants will sometimes be denoted by $C_{univ}$ and $C$. These may change from line to line unless otherwise specified. Sometimes in this paper, we will write $X\lesssim Y\leq Z$. This means that there exists a positive universal constant $C$ such that $X\leq CY\leq Z.$ In this paper we will sometimes refer to a quantity (for example ${M}$) being `sufficiently large', which should be understood as ${M}$ being larger than some universal constant that can (in principle) be specified.
For a vector $a$, $a_{i}$ denotes the $i^{th}$ component of $a$. We will sometimes write $a=(a',a_{3})$ instead of $a=(a_1,a_2,a_3)$.
For $(x,t)\in\R^{3+1}$ and $r>0$, 
we write $B(x,r):= \{y\in\R^3; |x-y|<r\}$, 
$Q((x,t),r):= B(x,r)\times (t-r^2,t)$ and $Q(r):=Q((0,0),r)$.
For $a,\, b\in\R^3$, we write $(a\otimes b)_{\alpha\beta}=a_\alpha b_\beta$, and for $A,\, B\in M_3(\R)$, $A:B=A_{\alpha\beta}B_{\alpha\beta}$. Here and in the whole paper we use Einstein's convention on repeated indices.  For $F:\Omega\subseteq\mathbb{R}^3\rightarrow\mathbb{R}^3$, we define $\nabla F\in M_{3}(\mathbb{R})$ by $(\nabla F(x))_{\alpha\beta}:= \partial_{\beta} F_{\alpha}$. 

For ${\lambda}\in \mathbb{R}$, $\floor*{\lambda}$ denotes the greatest integer less than $\lambda$. Furthermore, $\ceil*{\lambda}$ denotes the smallest integer greater than $\lambda$.

\subsubsection{Function spaces and norms}

For $\mathcal{O}\subseteq \R^3$, $k\in \mathbb{N}$ and $1\leq p\leq \infty$, $W^{k,p}(\mathcal{O})$ denotes the Sobolev space with differentiability $k$ and integrability $p$.
  
If $X$ is a Banach space with norm $\|\cdot\|_{X}$, then $L^{s}(a,b;X)$, with $a<b$ and $s\in[1,\infty)$,  will denote the usual Banach space of strongly measurable $X$-valued functions $f(t)$ on $(a,b)$ such that
$$\|f\|_{L^{s}(a,b;X)}:=\left(\int\limits_{a}^{b}\|f(t)\|_{X}^{s}dt\right)^{\frac{1}{s}}<+\infty.$$ 
The usual modification is made if $s=\infty$. Let $C([a,b]; X)$ denote the space of continuous $X$ valued functions on $[a,b]$ with usual norm. In addition, let $C_{w}([a,b]; X)$ denote the space of $X$ valued functions, which are continuous from $[a,b]$ to the weak topology of $X$. 

Let $\mathcal{O}\subseteq\R^n$ and $a,b\in\mathbb{R}$. Sometimes we will denote $L^{p}(a,b; L^{q}(\mathcal{O}))$ and $L^{p}(a,b; W^{k,q}(\mathcal{O}))$ by $L^{p}_{t}L^{q}_{x}(\mathcal{O}\times (a,b))$ and $L^{p}_{t}W^{k,q}_{x}(\mathcal{O}\times (a,b)).$ For the case $p=q$ we will often write $L^{p}_{x,t}(\mathcal{O}\times (a,b))$ to denote $L^{p}(a,b; L^{p}(\mathcal{O})$. For convenience, we will often use the following notation for energy type norms
\begin{equation}\label{energynormdef}
\|f\|_{L^{\infty}_{t}L^{2}_{x}\cap L^{2}_{t}\dot{H}^{1}_{x}(\mathcal{O}\times (a,b))}:=\Big(\|f\|_{L^{\infty}(a,b; L^{2}(\mathcal{O}))}^2+\|\nabla f\|^2_{L^{2}(a,b; L^{2}(\mathcal{O}))}\Big)^{\frac{1}{2}}.
\end{equation}
\subsubsection{Approximate axisymmetric functions}
\begin{definition}[Approximately axisymmetric functions]\label{approxaxi}
We say that $f:\R^3\rightarrow\R^3$ is \textit{approximately axisymmetric} if there exists an axisymmetric function $\phi:\R^3\rightarrow\R$ and a universal constant $C\in [1,\infty)$ such that
\begin{equation}\label{eqn:approxaxi}
C^{-1}\phi(x)\leq |f(x)|\leq C\phi(x)\qquad\textrm{for}\,\,\textrm{all}\quad x\in\R^3.
\end{equation}  
\end{definition}
\subsubsection{Notions of solutions and singularity}
\begin{definition}[Suitable weak solution]\label{sws}
Let $\Omega\subseteq\R^3$. We say that $(v,\pi)$ is a \textit{suitable weak solution} to the Navier-Stokes equations \eqref{eq:NSE} 
in $\Omega\times (T_{1},T)$ if it fulfills the properties described in \cite{gregory2014lecture} (Definition 6.1 p.133 in \cite{gregory2014lecture}).
\end{definition}
\begin{definition}[Singular point]\label{singularpoint}
Let $(v,\pi)$ be a suitable weak solution to the Navier-Stokes equations in $\Omega\times (T_{1},T)$. We say $(y,s)\in \bar{\Omega}\times (T_1,T]$ is a \textit{singular point} of $v$ if $v\notin L^{\infty}(Q((y,s),r)\cap (\Omega\times (T_{1},T)) )$ for all sufficiently small $r$.
\end{definition}
\begin{definition}[Leray-Hopf solution]\label{Lerays}
We say that $(v,\pi)$ is a \textit{Leray-Hopf solution} to the Navier-Stokes equations on $\R^3\times (T_2,\infty)$
  if 
\begin{itemize}
\item $v\in C_{w}([T_2,\infty); L^{2}_{\sigma}(\R^3))\cap L^{2}(T_2,\infty; \dot{H}^{1}(\R^3))$, 
\item $v$ is a solution to \eqref{eq:NSE} in the sense of distributions,
\item it satisfies the global energy inequality
\begin{equation}\label{energyinequalityLeray}
\|v(\cdot,t)\|_{L^{2}}^2+2\int\limits_{T_2}^{t}\int\limits_{\mathbb{R}^3}|\nabla v|^2 dyds\leq \|v(\cdot,0)\|_{L^{2}}^2\,\,\,\,\textrm{for}\,\,\textrm{all}\,\,t\in [T_2,\infty).
\end{equation}
\end{itemize}
Here, $L^{2}_{\sigma}(\R^3)$ denotes the closure of smooth divergence-free compactly supported functions with respect to the $L^2(\R^3)$ norm.
\end{definition}
\begin{definition}[Suitable Leray-Hopf solution]\label{sLerays}
We say that $(v,\pi)$ is a \textit{suitable Leray-Hopf solution} to the Navier-Stokes equations on $\R^3\times (0,\infty)$
  if
  \begin{itemize}
  \item $v$ is a weak Leray-Hopf solution on $\R^3\times (0,\infty)$,
  \item
  $(v,\pi)$ is a suitable weak solution on $\R^3\times (0,\infty)$. 
  \end{itemize} 

\end{definition}

\section{Improved quantitative annulus of regularity}
For obtaining the lower bounds on the solution in Theorems \ref{thm:loweratblowup}-\ref{thm:lowernoref}, we will make use of the following energy bound for solutions with $L^{3}(\R^3)$ initial data from \cite{TB24}. In particular, this energy bound is improved compared to the exponential bound used in \cite{BP21} to obtain the triple exponentially small lower bound \eqref{eq:BP21lower} on the blow-up profile.
\begin{lemma}\label{energyboundweakL3}(Improved energy bounds)
Let $v:\R^3\times (0,\infty)\rightarrow\R^3$ be a suitable weak Leray-Hopf solution to the Navier-Stokes equations, with initial data $v_{0}$ and associated pressure $\pi:\R^3\times (0,\infty)\rightarrow \R$.

Suppose that
\begin{equation}\label{eq:initialdataweakL3}
\|v_{0}\|_{L^{3}(\R^3)}\leq M.
\end{equation}
Define
$$ u(\cdot,s):= v(\cdot,s)-e^{s\Delta}v_{0},$$
where $e^{t\Delta}$ is the heat semi-group.

Then for $M\geq 1$, the above assumptions imply that for $t>0$
\begin{equation}\label{eq:perturbationenergyest}
\|u\|^{2}_{L^{\infty}_{t}L^{2}_{x}\cap L^{2}_{t}\dot{H}^{1}_{x}(\R^3\times (0,t))}\lesssim M^{12}t^{\frac{1}{2}},
\end{equation}
\begin{equation}\label{L103}
\|v\|_{L^{\frac{10}{3}}(\R^3\times (0,t))}\lesssim M^{6}t^{\frac{1}{4}}\quad\textrm{and}\quad \|\pi\|_{L^{\frac{5}{3}}(\R^3\times (0,t))}\lesssim M^{12}t^{\frac{1}{2}}.
\end{equation} 

Moreover for any $x_0\in \R^3$ and $R>0$ we have
\begin{equation}\label{eq:linearenergy}
 \|e^{t\Delta}v_0\|^{2}_{L^{\infty}_{t}L^{2}_{x}\cap L^{2}_{t}\dot{H}^{1}_{x}(B(x_0,4R)\times (0,R^2))}\lesssim RM^2.
\end{equation}
\end{lemma}
\begin{proof}
First, note that \eqref{eq:perturbationenergyest}-\eqref{L103} can be obtained from \cite[Lemma 1]{TB24} using the embedding $L^{3}(\R^3)\hookrightarrow L^{3,\infty}(\R^3)$ and an appropriate rescaling. Though in \cite[Lemma 1]{TB24} it is assumed that $v$ is bounded, this regularity is unnecessary. In particular, \cite[Lemma 1]{TB24} is proven using \cite[Lemma 3.4]{BSS18}, which can be verified to hold for suitable weak Leray-Hopf solutions.\\ 
Finally, to derive \eqref{eq:linearenergy}, let $\phi\in C^{\infty}(\R^3\times \R; [0,1])$ be a test function such that
\begin{itemize}
\item $\phi=1$ on $B(x_0,4R)\times [-R^2, R^2]$,
\item $\supp\phi \subset B(x_0,8R)\times [-4R^2, 4R^2]$,
\item $|\partial_{s}\phi(x,s)|+|\nabla^2 \phi(x,s)|\lesssim\frac{1}{R^2}$ for all $(x,s)\in B(x_0,8R)\times [-4R^2, 4R^2]$.
\end{itemize}
As $e^{t\Delta}v_0$ solves the heat equation, multiplying the heat equation for $e^{t\Delta}v_0$  by $e^{t\Delta}v_0\phi$ and integrating by parts readily gives that for $t>0$
\begin{equation*}
\begin{split}
&\int\limits_{\R^3}|e^{t\Delta}v_0|^2\phi(x,t) dx+2\int\limits_{0}^{t}\int\limits_{\R^3} |\nabla e^{s\Delta}v_0|^2 \phi dxds\\
&=\int\limits_{\R^3}|v_0|^2\phi(x,0) dx+\int\limits_{0}^{t}\int\limits_{\R^3} | e^{s\Delta}v_0|^2 (\partial_{s}\phi+\Delta\phi) dxds.
\end{split}
\end{equation*}
From this, H\"{o}lder's inequality and $\|e^{t\Delta} v_0\|_{L^{3}(\R^3)}\leq M$ we obtain
\begin{equation*}
\begin{split}
& \|e^{t\Delta}v_0\|^{2}_{L^{\infty}(0,R^2; L^{2}(B(x_0,4R)))}+\int\limits_{0}^{R^2}\int\limits_{B(x_0,4R)} |\nabla e^{s\Delta}v_0|^2 dxds\\
&\lesssim \int\limits_{B(x_0, 8R)}|v_0|^2 dx+\frac{1}{R^2}\int\limits_{0}^{R^2}\int\limits_{B(x_0, 8R)} | e^{s\Delta}v_0|^2 dxds\lesssim RM^2.
\end{split}
\end{equation*}

\end{proof}
In order to obtain the improved quantitative region of regularity for solutions with approximately axisymmetric initial data, we use previous local-in-space\\ smoothing results for the Navier-Stokes equations. 
In particular, we essentially use Theorem 1.1, Remark 1.2 and Theorem 3.1 of \cite{KMT21}. The minor difference with the statement below is that the localized smallness of the initial data gives small local bounds on the solution.
\begin{proposition}(Local-in-space short time smoothing, \cite[Theorem 1.1, Remark 1.2 and Theorem 3.1]{KMT21})\label{KMT}\\
There exists  positive universal constants $c_0$ and $\epsilon_{*}$ such that the following holds true.\\ Let $N\geq 1$. Suppose that $(v,\pi)$ is a suitable weak solution to the Navier-Stokes equations on $B(0,4)\times (0,T)$ such that
\begin{equation}\label{vi.dKMT}
\lim_{t\rightarrow 0^+}\|v(\cdot,t)-v_0\|_{L^{2}(B(0,4))}=0,\quad\textrm{with}\,\, v_0\in L^{2}(B(0,4)).
\end{equation}
Furthermore suppose that
\begin{equation}\label{venergyMKMT}
\|v\|^{2}_{L^{\infty}_{t}L^{2}_{x}\cap L^{2}_{t}\dot{H}^{1}( B(0,4)\times (0,T))}+\|\pi\|_{L^{\frac{3}{2}}_{x,t}(B(0,4)\times (0,T))}\leq N
\end{equation}
and
\begin{equation}\label{L3smallKMT}
\|v_0\|_{L^{3}(B(0,3))}\leq \epsilon_{0}\leq \epsilon_*.
\end{equation}
Then the above hypothesis imply that $v$ and $\omega=\nabla\times v$ 
 satisfy the quantitative bounds
\begin{equation}\label{quantboundedKMT}
\begin{split}
&t^{\frac{j+1}{2}}|\nabla^{j}v(x,t)|\lesssim\epsilon_{0}^{\frac{1}{3}}\quad\textrm{and}\quad t^{\frac{3}{2}}|\nabla\omega(x,t)|\lesssim\epsilon_{0}^{\frac{1}{3}}, \quad\textrm{where}\quad j=0,1\quad\textrm{and}\\
& (x,t)\in B(0,2)\times (0, \min(T, c_0\epsilon_{0}^{12}N^{-18})).
\end{split}
\end{equation}
\end{proposition}
\begin{remark}(Sketch of proof)\label{KMTsketch}\\
 By means of translation, it is sufficient to show that if 
\begin{equation}\label{vi.dKMTWLOG}
\lim_{t\rightarrow 0^+}\|v(\cdot,t)-v_0\|_{L^{2}(B(0,2))}=0,\quad\textrm{with}\,\, v_0\in L^{2}(B(0,2)),
\end{equation}
\begin{equation}\label{venergyMKMTWLOG}
\|v\|^{2}_{L^{\infty}_{t}L^{2}_{x}\cap L^{2}_{t}\dot{H}^{1}( B(0,2)\times (0,T))}+\|\pi\|_{L^{\frac{3}{2}}_{x,t}(B(0,2)\times (0,T))}\leq N\quad\textrm{and}
\end{equation}
\begin{equation}\label{L3smallKMTWLOG}
\|v_0\|_{L^{3}(B(0,1))}\leq \epsilon_{0}\leq \epsilon_*,
\end{equation}
then (for small enough $\epsilon_*$) the above hypothesis imply that $v$ and $\omega=\nabla\times v$ 
 satisfy the quantitative bounds
\begin{equation}\label{quantboundedKMTWLOG}
\begin{split}
&t^{\frac{j+1}{2}}|\nabla^{j}v(0,t)|\lesssim\epsilon_{0}^{\frac{1}{3}}\quad\textrm{and}\quad t^{\frac{3}{2}}|\nabla\omega(0,t)|\lesssim\epsilon_{0}^{\frac{1}{3}}, \quad\textrm{where}\quad j=0,1\quad\textrm{and}\\
& t\in (0, \min(T, c_0\epsilon_{0}^{12}N^{-18})).
\end{split}
\end{equation}
We briefly sketch how to obtain \eqref{quantboundedKMTWLOG}, using the arguments in \cite{KMT21}.
Note that for each $R\in (0,1)$, we have that \eqref{L3smallKMTWLOG} implies that
\begin{equation}\label{eq:morreyembed}
N_{R}:=\sup_{R\leq r\leq 1} \frac{1}{r}\int\limits_{B(0,r)} |v_0|^2 dx\lesssim \epsilon_0\leq\epsilon_*.
\end{equation}
Then the same reasoning in \cite{KMT21} implies that
\begin{equation}\label{eq:CKNKMTbound}
\frac{1}{t}\int\limits^{t}_{0}\int\limits_{B(0,t^{\frac{1}{2}})} |v|^3+|p|^{\frac{3}{2}} dxds\lesssim \epsilon_{0}\quad\forall t\in (0, \min(T, c_0\epsilon_{0}^{12}N^{-18})).
\end{equation}
Then for small enough $\epsilon_{*}$ (note that $\epsilon_0\leq \epsilon_{*}$), \eqref{eq:CKNKMTbound} enables us to apply \cite[Proposition 6.4]{BP21} for each fixed $t\in (0, \min(T, c_0\epsilon_{0}^{12}N^{-18}))$, which gives the desired conclusion \eqref{quantboundedKMTWLOG}.
\end{remark}
Next we will use the previous Proposition on local-in-space smoothing to obtain an improved quantitative region of regularity for solutions with approximately axisymmetric initial data. This plays a crucial role in obtaining the claimed lower bounds in Theorems \ref{thm:loweratblowup}-\ref{thm:lowernoref}.
\begin{proposition}\label{prop:improvedannulus}
Let $p\in (3,\infty]$. There exists a positive universal constant $M_2>1$ such that the following statement holds true.\\
Let $v:\R^3\times (0,\infty)\rightarrow \R^3$ be a suitable weak Leray-Hopf solution, with initial data $v_{0}:\R^3\rightarrow\R^3$.
Suppose that $v_{0}(x_1,x_2,x_3)$ is approximately axisymmetric and there exists $M\geq M_{2}$ such that
\begin{equation}\label{eq:initialdataassumption3}
\|v_{0}\|_{L^3(\R^3)}+\||x_3|^{1-\frac{3}{p}}|v_{0}(\cdot,x_3)|\|_{L^{p}(\R^3)}\leq M.
\end{equation}
Then the above assumptions imply that for every $T>0$ and $j=0,1$, we have that $v$ and $\omega=\nabla\times v$ satisfy the bounds
\begin{equation}\label{eq:improvedannulus}
\begin{split}
&\sup_{\Big\{(x,t): |x|\geq M^{700+\frac{30}{1-\frac{3}{p}}}T^{\frac{1}{2}},\,\,t\in (0,T)\Big\}} t^{\frac{1+j}{2}} |\nabla^{j}v(x,t)|\lesssim M^{-9}\quad\textrm{and}\\
&\sup_{\Big\{(x,t): |x|\geq M^{700+\frac{30}{1-\frac{3}{p}}}T^{\frac{1}{2}},\,\,t\in (0,T)\Big\}} t^{\frac{3}{2}} |\nabla\omega(x,t)|\lesssim M^{-9}.
\end{split}
\end{equation}
\end{proposition}
\begin{proof}
By means of the rescaling $v_{T}(x,t):=T^{\frac{1}{2}} v(T^{\frac{1}{2}}x, Tt)$ and the scaling invariance of \eqref{eq:initialdataassumption3}-\eqref{eq:improvedannulus}, we may assume without loss of generality that $T=1$.
In the rest of the proof, fix $x=(x',x_3)$ with $|x|\geq M^{700+\frac{30}{1-\frac{3}{p}}}$.
Furthermore, using Lemma \ref{energyboundweakL3} and \eqref{eq:initialdataassumption3}, we have that
\begin{equation}\label{eq:boundsolM600}
\|v\|^{2}_{L^{\infty}_{t}L^{2}_{x}\cap L^{2}_{t}\dot{H}^{1}(B(x,4M^{600})\times (0,1))}\lesssim M^{602}.
\end{equation}
Using H\"{o}lder's inequality and \eqref{L103} we also have
\begin{equation}\label{eq:presboundM600}
\|\pi\|_{L^{\frac{3}{2}}_{x,t}(B(x,4M^{600})\times (0,1))}\lesssim M^{132}.
\end{equation}

Next using \eqref{eq:initialdataassumption3} and that $v_0$ is approximately axisymmetric, we claim that
\begin{equation}\label{eq:boundidM600}
\int\limits_{B(x=(x',x_3),3M^{600})}|v_{0}(y)|^3 dy\lesssim M^{-87}.
\end{equation}
\textbf{Case (a)}: $|x'|\geq \tfrac{1}{2}M^{700+\frac{30}{1-\frac{3}{p}}}$\\
In this case there exists $\frac{C_{univ} |x'|}{M^{600}}$ disjoint balls $$\{B(({y'_{(j)}}, x_3), 3M^{600})\}_{1\leq j\leq \frac{C_{univ} |x'|}{M^{600}}}$$ with $|y'_{(j)}|=|x'|$. Thus,
\begin{equation}\label{eq:v0disjointL3}
\sum_{j=1}^{j=\frac{C_{univ} |x'|}{M^{600}}} \int\limits_{B((y'_{(j)}, x_3), 3M^{600})}|v_{0}(y)|^3 dy\leq M^3
\end{equation} Recall that as $v_0$ is approximately axisymmetric, there exists an axisymmetric function $\phi:\R^3\rightarrow\R$ and a universal constant $C\in [1,\infty)$ such that
\begin{equation}\label{eqn:approxaxiv0}
C^{-1}\phi(y)\leq |v_0(y)|\leq C\phi(y)\qquad\textrm{for}\,\,\textrm{all}\quad x\in\R^3.
\end{equation} 
Using this and \eqref{eq:v0disjointL3}, we get
\begin{equation*}
\begin{split}
&\int\limits_{B((x',x_3), 3M^{600})}|v_{0}(y)|^3 dy\lesssim\int\limits_{B((x',x_3), 3M^{600})}|\phi(y)|^3 dy\\
&=\frac{M^{600}}{C_{univ}|x'|} \sum_{j=1}^{j=\frac{C_{univ} |x'|}{M^{600}}} \int\limits_{B((y'_{(j)}, x_3), 3M^{600})}|\phi(y)|^3 dy\\
&\lesssim \frac{M^{600}}{|x'|}\sum_{j=1}^{j=\frac{C_{univ} |x'|}{M^{600}}} \int\limits_{B((y'_{(j)}, x_3), 3M^{600})}|v_{0}(y)|^3 dy\lesssim \frac{M^{603}}{|x'|}\lesssim M^{-97}.
\end{split}
\end{equation*}
\textbf{Case (b)}: $|x_3|\geq \tfrac{1}{2}M^{700+\frac{30}{1-\frac{3}{p}}}$\\
In this case, for $M$ sufficiently large (independent of $p$)
$$B((x',x_3), M^{600})\subset \{(y',y_{3})\in \R^3: |y_{3}|\geq M^{600+\frac{30}{1-\frac{3}{p}}}\}.$$
Using this and \eqref{eq:initialdataassumption3}, we get
$$\|v_{0}\|_{L^{p}(B((x',x_3), 3M^{600}))}\leq {M^{-(600(1-\frac{3}{p})+29)}}. $$
Using this and H\"{o}lder's inequality, we have 
$$\int\limits_{B(x=(x',x_3),3M^{600})}|v_{0}(y)|^3 dy\leq (150M^{1800})^{1-\frac{3}{p}}\|v_{0}\|_{L^{p}(B((x',x_3), 3M^{600}))}^3\lesssim M^{-87}. $$
Hence, in all cases we arrive at \eqref{eq:boundidM600}.

Let $\lambda=M^{600}$ and define the rescaled functions $v_{\lambda}:\R^3\times (0,\infty)\rightarrow\R^3$, $\pi_{\lambda}:\R^3\times (0,\infty)\rightarrow\R$ and $v_{0\lambda}:\R^3\rightarrow \R$ by
\begin{equation}\label{eq:rescaledM600}
(v_{\lambda}, \pi_{\lambda})(y,s):=(\lambda v, \lambda^2 \pi)(\lambda y+x,\lambda^2 s)\quad\textrm{and}\quad v_{0\lambda}(y):=\lambda v_{0}(\lambda y+x).
\end{equation}
From \eqref{eq:boundsolM600}-\eqref{eq:boundidM600}, we infer that for $M$ sufficiently large
\begin{equation}\label{eq:boundsolrescaled}
\|v_{\lambda}\|^{2}_{L^{\infty}_{t}L^{2}_{x}\cap L^{2}_{t}\dot{H}^{1}(B(0,4)\times (0,M^{-1200}))}+\|\pi_{\lambda}\|_{L^{\frac{3}{2}}_{x,t}(B(0,4)\times (0,M^{-1200}))}\leq M^3
\end{equation}
and
\begin{equation}\label{eq:boundidrescaled}
\|v_{0\lambda}\|_{L^{3}(B(0,3))}\leq M^{-27}.
\end{equation}
Using \eqref{eq:boundsolrescaled}-\eqref{eq:boundidrescaled}, we can then apply Proposition \ref{KMT} with $\epsilon_{0}=M^{-27}$, $N=M^{3}$ and $T=M^{-1200}$. Then \eqref{quantboundedKMT} gives that for $M$ sufficiently large we have that for $j=0,1$ 
\begin{equation}\label{quantboundedKMTrescaled}
\begin{split}
&s^{\frac{j+1}{2}}|\nabla^{j}v_{\lambda}(0,s)|\lesssim M^{-9} \quad\textrm{and}\quad s^{\frac{3}{2}}|\nabla \omega_{\lambda}(0,s)|\lesssim M^{-9},\quad\textrm{where}\\
& s\in (0, \min(T, c_0\epsilon_{0}^{12}N^{-18})).
\end{split}
\end{equation} 
Noting that for $M$ sufficiently large $$(0, \min(T, c_0\epsilon_{0}^{12}N^{-18}))=(0,\min(M^{-1200}, c_0M^{-378}))=(0,M^{-1200}),$$ we obtain that for $j=0,1$
\begin{equation*}
\begin{split}
\sup_{s\in (0,M^{-1200})}s^{\frac{j+1}{2}}|\nabla^{j}v_{\lambda}(0,s)|\lesssim M^{-9} \quad\textrm{and}\quad \sup_{s\in (0,M^{-1200})}s^{\frac{3}{2}}|\nabla \omega_{\lambda}(0,s)|\lesssim M^{-9}.
\end{split}
\end{equation*}
From this and the rescaling \eqref{eq:rescaledM600}, we infer that for $j=0,1$ $$\sup_{t\in (0,1)}t^{\frac{j+1}{2}}|\nabla^{j}v(x,t)|\lesssim M^{-9} \quad\textrm{and}\quad \sup_{t\in (0,1)}t^{\frac{3}{2}}|\nabla \omega(x,t)|\lesssim M^{-9}$$
where $|x|\geq M^{700+\frac{30}{1-\frac{3}{p}}}$ is arbitrary. This gives the desired conclusion for $T=1$, as required.

\end{proof}
\section{Quantitative regions of concentration and epochs of regularity}
As in \cite[Proposition 4.4]{BP21}, a key component in the strategy to prove Theorems \ref{thm:loweratblowup}-\ref{thm:lowernoref} is to quantify (localized) vorticity concentration induced by a singular point. For this purpose, the following statement from \cite{TB23} is convenient.
\begin{lemma}{(Local-in-space short time smoothing, with initial vorticity locally in $L^2$) \cite[Corollary 2  and Remark 7]{TB23}}\label{localinspacevort}\\There exists a  positive universal constant $c_2$  such that the following holds true.\\Let $N$ be sufficiently large. Suppose that $(v,\pi)$ is a suitable weak solution to the Navier-Stokes equations on $B(0,4)\times (0,T)$ such that
\begin{equation}\label{vi.dvort}
\lim_{t\rightarrow 0^+}\|v(\cdot,t)-v_0\|_{L^{2}(B(0,4))}=0,\quad\textrm{with}\,\, v_0\in W^{1,2}(B(0,4)).
\end{equation}
Furthermore suppose that
\begin{equation}\label{venergyvort}
\|v\|^{2}_{L^{\infty}_{t}L^{2}_{x}\cap L^{2}_{t}\dot{H}_{x}^{1}(B(0,4)\times (0,T))}+\|\pi\|_{L^{\frac{3}{2}}_{x,t}(B(0,4)\times (0,T))}\leq N
\end{equation}
and that for $\omega_0=\nabla\times v_0$
\begin{equation}\label{vorti.d}
\|\omega_0\|_{L^{2}(B(0,\frac{7}{2}))}^2\leq N.
\end{equation}
Then the above hypothesis imply that $v$
  satisfies the quantitative bound
\begin{equation}\label{quantboundedvort}
|\nabla^{j}v(x,t)|\lesssim_{j} t^{-\frac{j+1}{2}}\quad\textrm{where}\,\,j=0,1\ldots\textrm{and}\quad (x,t)\in B(0,2)\times (0,\min(T,c_{2}N^{-44})).
\end{equation}
\end{lemma}
\begin{proposition}\label{pro:quantregionconc}{(Improved quantitative regions of vorticity concentration)}\\
Let $p\in (3,\infty]$ and let $\lambda_{0}-\lambda_{1}$ be as in \eqref{eq:lambdaolambda1def}. There exists a positive universal constant $M_3>1$ such that the following statement holds true.\\
Let $v:\R^3\times (0,\infty)\rightarrow \R^3$ be a suitable weak Leray-Hopf solution, with initial data $v_{0}:\R^3\rightarrow\R^3$, that has a singular point at $t=T^*$.
Suppose that $v_{0}(x_1,x_2,x_3)$ is approximately axisymmetric and there exists $M\geq M_{3}$ such that
\begin{equation}\label{initialdataassumption4}
\|v_{0}\|_{L^3(\R^3)}+\||x_3|^{1-\frac{3}{p}}|v_{0}(\cdot,x_3)|\|_{L^{p}(\R^3)}\leq M.
\end{equation}
Then  the above assumptions imply that there exists a set $\Sigma\subset [0,T^*]$ of full Lebesgue measure $|\Sigma|=T^*$ such that $\|\nabla v(\cdot,t)\|_{L^{2}(\R^3)}<\infty$ for all $t\in\Sigma$ with
\begin{equation}\label{eq:vortconcnearblowup}
\int\limits_{B(0, M^{701+\frac{30}{1-\frac{3}{p}}}t^{\frac{1}{2}})} |\omega(x,t)|^2 dx\geq \frac{M^{-297}}{t^{\frac{1}{2}}}\qquad \forall t\in [(1-2\lambda_{1})T^*, (1-\lambda_{1})T^*]\cap\Sigma
\end{equation}
and 
\begin{equation}\label{eq:vortconcawayfrominitial}
\int\limits_{B(0, M^{708+\frac{30}{1-\frac{3}{p}}}t^{\frac{1}{2}})} |\omega(x,t)|^2 dx\geq \frac{M^{-303}}{t^{\frac{1}{2}}}\qquad \forall t\in [M^{-13}T^*, (1-\lambda_{1})T^*]\cap\Sigma
\end{equation}
\end{proposition}
\begin{proof}
The strategy below follows \cite[Proposition 4.4]{BP21}, specifically using a rescaled version of local-in-space smoothing to infer that the singular point induces localized vorticity concentration by contraposition. By comparison the vorticity concentration in \cite[Proposition 4.4]{BP21} is
\begin{itemize}
\item(i) centered on the spatial coordinate of the singular point,
\item(ii) exponentially small with respect to the initial data and with the radius of the ball for the concentration being exponentially large.
\end{itemize}
Point (i) is undesirable and we are able to remove any reference to the singular point using the improved region of regularity (Proposition \ref{prop:improvedannulus}). Point (ii) would likely lead to much worse bounds in Theorems \ref{thm:loweratblowup}-\ref{thm:lowernoref}. For \eqref{eq:vortconcnearblowup}-\eqref{eq:vortconcawayfrominitial}, we improve upon (ii) by implementing the improved energy estimate in Lemma \ref{energyboundweakL3}.

Define
\begin{equation}\label{eq:singmasetdef}
\Sigma:=\{t_{0}\in [0,T^*]: \lim_{s\rightarrow t_{0}^{+}}\|v(\cdot,s)-v(\cdot,t_{0})\|_{L^{2}(\R^3)}=0\,\,\textrm{and}\,\,\|\nabla v(\cdot,t_{0})\|_{L^{2}(\R^3)}<\infty\}.
\end{equation}
It is known that for a suitable weak Leray-Hopf solution, $\Sigma$ has full measure in $[0,T^*]$. 

Let us first prove \eqref{eq:vortconcnearblowup} by fixing $t\in [(1-2\lambda_{1})T^*, (1-\lambda_{1})T^*]\cap\Sigma$.
For this, we have
\begin{equation}\label{eq:distancefromTstar}
 \frac{\lambda_{1}t}{1-\lambda_1}\leq \lambda_1 T^*\leq T^*-t\leq 2\lambda_1 T^*\leq  \frac{2\lambda_1 t}{1-2\lambda_1}.
\end{equation}
Let $(x_0,T^*)$ be a singular point. From \eqref{eq:improvedannulus}, we have that
\begin{equation}\label{eq:singpointbound}
|x_0|< M^{700+\frac{30}{1-\frac{3}{p}}}(T^*)^{\frac{1}{2}}.
\end{equation}
From \eqref{eq:distancefromTstar}-\eqref{eq:singpointbound}, we see that for $M$ sufficiently large (independent of $p$), to prove \eqref{eq:vortconcnearblowup} it is sufficient to show that
\begin{equation}\label{eq:vortconcreduction}
\int\limits_{B(x_0, \tfrac{7}{2}M^{310}(T^*-t)^{\frac{1}{2}})} |\omega(x,t)|^2 dx\geq \frac{M^{-296}}{(T^*-t)^{\frac{1}{2}}}.
\end{equation}
To establish this, it is sufficient show that 
\begin{equation}\label{eq:vortconcentrationcontrapositive}
\int\limits_{B(x_0, \tfrac{7}{2}M^{310}(T^*-t)^{\frac{1}{2}})} |\omega(x,t)|^2 dx< \frac{M^{-296}}{(T^*-t)^{\frac{1}{2}}}
\end{equation}
implies $(x_0,T^*)$ {is not a singular point}.

From Lemma \ref{energyboundweakL3}, H\"{o}lder's inequality and 
\begin{equation}\label{eq:blowupcfdistance}
\lambda_{1}T^*\leq T^*-t, 
\end{equation}
we have that for $M$ sufficiently large
\begin{equation}\label{eq:velocityboundconc}
\|v\|^{2}_{L^{\infty}_{t}L^{2}_{x}\cap L^{2}_{t}\dot{H}^{1}(B(x_0,4M^{310}(T^*-t)^{\frac{1}{2}})\times (t,T^*))}\lesssim M^{312}(T^*-t)^{\frac{1}{2}}\quad\textrm{and}
\end{equation}
\begin{equation}\label{eq:pressureboundconc}
\|\pi\|_{L^{\frac{3}{2}}_{x,t}(B(x_0,4M^{310}(T^*-t)^{\frac{1}{2}})\times (t,T^*))}\lesssim M^{74}(T^*-t)^{\frac{2}{3}}.
\end{equation}
Let $\lambda=M^{310}(T^*-t)^{\frac{1}{2}}$ and define the rescaled functions $v_{\lambda}:\R^3\times (0,\infty)\rightarrow\R^3$, $\pi_{\lambda}:\R^3\times (0,\infty)\rightarrow\R$ and $v_{0\lambda}:\R^3\rightarrow \R$ by
\begin{equation}\label{eq:rescaledM310}
(v_{\lambda}, \pi_{\lambda})(y,s):=(\lambda v, \lambda^2 \pi)(\lambda y+x_0,\lambda^2 s+t)\quad\textrm{and}\quad v_{0\lambda}(x)=\lambda v(\lambda y+x_0,t).
\end{equation}
From \eqref{eq:vortconcentrationcontrapositive} and \eqref{eq:velocityboundconc}-\eqref{eq:pressureboundconc}, we infer that for $M$ sufficiently large
\begin{equation}\label{eq:boundsolrescaledconc}
\|v_{\lambda}\|^{2}_{L^{\infty}_{t}L^{2}_{x}\cap L^{2}_{t}\dot{H}^{1}(B(0,4)\times (0,M^{-620}))}+\|\pi_{\lambda}\|_{L^{\frac{3}{2}}_{x,t}(B(0,4)\times (0,M^{-620}))}\leq M^3\leq M^{14}
\end{equation}
and for $\omega_{0\lambda}=\nabla\times v_{0\lambda}$
\begin{equation}\label{eq:boundidrescaledconc}
\|\omega_{0\lambda}\|^2_{L^{2}(B(0,\frac{7}{2}))}\leq M^{14}.
\end{equation}
Furthermore, as $t\in\Sigma$ we have
\begin{equation}\label{eq:contconc}
\lim_{s\rightarrow 0^{+}}\|v_{\lambda}(\cdot,s)-v_{0\lambda}\|_{L^{2}(\R^3)}=0.
\end{equation}
Using \eqref{eq:boundsolrescaledconc}-\eqref{eq:contconc}, we can then apply Lemma \ref{localinspacevort} with $N=M^{14}$ and $T=M^{-620}$. Then \eqref{quantboundedvort} gives that for $M$ sufficiently large 
\begin{equation}\label{quantboundedKMTrescaledconc}
\begin{split}
&|\nabla^{j}v_{\lambda}(0,s)|\lesssim s^{-\frac{j+1}{2}} \quad\textrm{where}\quad j=0,1\quad\textrm{and}\\
& s\in (0, \min(T, c_2 N^{-44}))
\end{split}
\end{equation} 
Noting that for $M$ sufficiently large $$(0, \min(T, c_2N^{-44}))=(0,\min(M^{-620}, c_{2}M^{-616}))=(0,M^{-620}),$$ we obtain that $(y,s)=(0, M^{-620})$ is not a singular point for $v_{\lambda}$. From this and the rescaling \eqref{eq:rescaledM310}, we infer that $(x_0,T^*)$ is not a singular point of $v$ as required.

Finally let us prove \eqref{eq:vortconcawayfrominitial} by fixing $t\in [M^{-13}T^*, (1-\lambda_{1})T^*]\cap\Sigma$. In this case, verbatim reasoning to the above gives
\begin{equation}\label{eq:vortconcrecall}
\int\limits_{B(x_0, \tfrac{7}{2}M^{310}(T^*-t)^{\frac{1}{2}})} |\omega(x,t)|^2 dx\geq \frac{M^{-296}}{(T^*-t)^{\frac{1}{2}}}.
\end{equation}
Taking into account \eqref{eq:singpointbound}, we have that
\begin{equation}\label{eq:parameters}
M^{13}t\geq T^*>T^*-t\quad\textrm{and}\quad |x_0|\leq M^{707+\frac{30}{1-\frac{3}{p}}}t^{\frac{1}{2}}.
\end{equation}
Combining \eqref{eq:vortconcrecall}-\eqref{eq:parameters} readily gives \eqref{eq:vortconcawayfrominitial}.
\end{proof}
We now use the improved energy estimate in Lemma \ref{energyboundweakL3} to obtain improved quantitative epochs of regularity.
\begin{proposition}\label{pro:epochL3}{(Improved quantitative epochs of regularity)}
Let $\lambda_{0}-\lambda_{1}$ be as in \eqref{eq:lambdaolambda1def}. There exists a positive universal constant $M_4>1$ such that the following statement holds true.\\
Let $v:\R^3\times (0,\infty)\rightarrow \R^3$ be a suitable weak Leray-Hopf solution, with initial data $v_{0}:\R^3\rightarrow\R^3$.
Suppose there exists $M\geq M_{4}$ such that
\begin{equation}\label{initialdataassumption5}
\|v_{0}\|_{L^3(\R^3)}\leq M.
\end{equation}
Then the above assumptions imply that for all $t>0$ there exists a  closed time interval $I'$ such that
\begin{equation}\label{eq:I'inclusionlength}
I'\subseteq [(1-2\lambda_1)t, (1-\tfrac{3}{2}\lambda_{1})t],\quad |I'|=M^{-126} t,
\end{equation}
\begin{equation}\label{eq:vsmoothepoch}
v\in C^{\infty}(\R^3\times I')\quad\textrm{and}
\end{equation}
\begin{equation}\label{eq:vboundI'}
\|\nabla^k v\|_{L^{\infty}_{x,t}(\R^3\times I')}\leq M^{-2}|I'|^{-\frac{k+1}{2}}\qquad\textrm{for}\quad k=0,1,2.
\end{equation}

\end{proposition}
\begin{proof}
The strategy follows \cite[Lemma 7.3 and Lemma 7.5]{BP21}. Namely, the quantitative epoch is obtained using the pigeonhole principle, weak-strong uniqueness and Caffarelli, Kohn and Nirenberg's $\varepsilon$-regularity criterion. Comparatively, the quantitative epoch in \cite[Lemma 7.5]{BP21} is exponentially shorter than \eqref{eq:I'inclusionlength} and the corresponding solution bound is exponentially larger than \eqref{eq:vboundI'}. We obtain improvements by implementing the improved energy estimate in Lemma \ref{energyboundweakL3} into the strategy in \cite[Lemma 7.3 and Lemma 7.5]{BP21}. A detailed sketch is provided below.

Let
\begin{equation}\label{eq:sigmadeft}
\begin{split}
\Sigma:=\Big\{&s_{0}\in [0,t]: \|v(\cdot,s)\|_{L^{2}(\R^3)}^2+2\int\limits_{s_0}^{s}\int\limits_{\R^3} |\nabla v(x,\theta)|^2 dxd\theta\leq \|v(\cdot,s_0)\|_{L^{2}(\R^3)}^2\\
&\forall s\in [s_0,\infty)\Big\}.
\end{split}
\end{equation}
As $\Sigma$ is known to be a full measure set in $[0,t]$, we can use \eqref{L103} and the pigeonhole principle to infer there exists 
\begin{equation}\label{eq:epochstartingpoint}
t_0\in \Sigma\cap [(1-2\lambda_1)t, (1-\tfrac{3}{2}\lambda_{1})t]\quad\textrm{with}\quad\|v(\cdot,t_0)\|_{L^{\frac{10}{3}}(\R^3)}\lesssim M^{6} t^{-\frac{1}{20}}.
\end{equation}
By local-in-time existence theory\footnote{See \cite{weissler} and \cite{Giga}, which are paraphrased in \cite[Proposition A.1]{BP21}.} for the Navier-Stokes equations, there exists a universal constant $C$ such that for $S_{mild}=C\|v(\cdot,t_0)\|_{L^{\frac{10}{3}}(\R^3)}^{-20}$ there exists $V\in C([0, S]; L^{\frac{10}{3}}(\R^3))$ such that 
\begin{equation}\label{eq:Vest}
\sup_{s\in [0,S_{mild}]}(\|V(\cdot,s)\|_{L^{\frac{10}{3}}(\R^3)}+s^{\frac{9}{20}}\|V(\cdot,s)\|_{L^{\infty}(\R^3)})\lesssim \|v(\cdot,t_0)\|_{L^{\frac{10}{3}}(\R^3)},
\end{equation}
\begin{equation}\label{eq:Vduhamel}
V(\cdot,s)= e^{s\Delta}v(\cdot,t_0)+\int\limits_{0}^{s} e^{(s-\theta)\Delta}\mathbb{P}\nabla\cdot(V\otimes V(\cdot,\theta)) \,d\theta\quad\forall s\in [0,S_{mild}].
\end{equation}
Here, $\mathbb{P}$ is the projection onto divergence-free vector fields.
Using \eqref{eq:Vest} and Lebesgue interpolation, we obtain
$$V\in L^{4}_{x,t}(\R^3\times (0,S_{mild})). $$
From this and \cite[Theorem 1.1]{galdi19}, we infer that
\begin{itemize}
\item $V\in C([0, S_{mild}]; L^{2}_{\sigma}(\R^3))\cap L^{2}(0,S_{mild}); \dot{H}^{1}(\R^3))$,
\item $\|V(\cdot,s)\|_{L^{2}(\R^3)}^{2}+2\int\limits_{0}^{s}\int\limits_{\R^3} |\nabla V(x,\theta)|^2 dxd\theta=\|v(\cdot,t_0)\|_{L^{2}(\R^3)}^2\,\,\forall s\in [0,S_{mild}].$
\end{itemize}
As $t_0\in \Sigma$, known weak-strong uniqueness results \cite{Le} imply that $v(\cdot, t)\equiv V(\cdot,t-t_0)$ on $\R^3\times [t_0, t_0+S_{mild}].$ Thus from \eqref{eq:epochstartingpoint}-\eqref{eq:Vest} and Calder\'{o}n-Zygmund estimates we have for $M$ sufficiently large
\begin{equation}\label{eq:vhigherintegrability}
\begin{split}
&\sup_{s\in[t_0, t_0+tM^{-121}]}\|v(\cdot,s)\|_{L^{\frac{10}{3}}(\R^3)}\lesssim M^{6}t^{-\frac{1}{20}}\quad\textrm{and}\\
&\sup_{s\in[t_0, t_0+tM^{-121}]}\|\pi(\cdot,s)\|_{L^{\frac{5}{3}}(\R^3)}\lesssim M^{12}t^{-\frac{1}{10}}.
\end{split}
\end{equation}
Standard bootstrap arguments yield that $v$ is smooth on $\R^3\times (t_0, t_0+tM^{-121}]$.
Following \cite[Lemma 7.3]{BP21}, an application of Caffarelli, Kohn and Nirenberg's $\varepsilon$-regularity criterion \cite{CKN} yields that for $k=0,1,2$
$$\|\nabla^k v\|_{L^{\infty}_{x,t}(\R^3\times[t_0+\frac{M^{-121}t}{2}, t_0+M^{-121}t])}\lesssim_{k} \Big(\frac{M^{-121}t}{2}\Big)^{-\frac{k+1}{2}}.$$
So for 
\begin{equation}\label{eq:I'def}
I':=[t_0+M^{-121}t-M^{-126}t, t_0+M^{-121}t]\subset \Big[t_0+\frac{M^{-121}t}{2}, t_0+M^{-121}t\Big]
\end{equation}
we get that $|I'|=M^{-126}t$ and  for $k=0,1,2$
$$\|\nabla^k v\|_{L^{\infty}_{x,t}(\R^3\times[t_0+\frac{M^{-121}t}{2}, t_0+M^{-121}t])}\lesssim_{k} \Big(\frac{M^{-121}t}{2}\Big)^{-\frac{k+1}{2}}\leq M^{-\frac{5}{2}}2^{\frac{3}{2}}|I'|^{-\frac{k+1}{2}}.$$
Thus, for $M$ sufficiently large, we obtain the desired conclusion.
\end{proof}
\section{Carleman inequalities and inductive quantitative backward uniqueness}
The quantitative Carleman inequalities we use are essentially those stated and proven by Tao in \cite[Proposition 4.2-4.3]{Ta21}, which involve $C^{\infty}$ functions. As our context is less regular, we use analogues given in \cite[Proposition 21-22]{BMT} that require less stringent regularity.
\begin{proposition}\label{pro:backuniqueness}{(Quantitative backward uniqueness Carleman inequality)}\\
\cite[Proposition 21]{BMT} 

Let 
 $T_1\in (0,\infty)$, $0<r_-<r_+<\infty$ and we define the space-time annulus 
\begin{equation*}
\mathcal A:=\{(x,t)\in\R^3\times\R\, :\ t\in[-T_1,0],\ r_-\leq |x|\leq r_+\}.
\end{equation*}
Let $w:\ \mathcal A\rightarrow\R^3$ be such that $w,\ \partial_tw,\ \nabla w$ and $\nabla^2w$ are continuous in space and time and such that $w$ satisfies the differential inequality 
\begin{equation}\label{e.diffineq}
|(\partial_{t}-\Delta) w(x,t)|\leq \frac{|w(x,t)|}{C_{Carl} T_1}+\frac{|\nabla w (x,t)|}{(C_{Carl} T_1)^{\frac{1}{2}}}\quad\mbox{on}\ \mathcal A\quad\textrm{with}\,\, C_{Carl}\in [9,\infty).
\end{equation}
Assume
\begin{equation}\label{e.lowerr-}
r_-^2\geq 4C_{Carl}T_1.
\end{equation}
Then we have the following bound 
\begin{equation}\label{e.conclcarlone}
\int\limits_{-\frac{T_1}{4}}^{0}\int\limits_{10r_-\leq|x|\leq \frac{r_+}2}(T_{1}^{-1}|w|^2+|\nabla w|^2)\, dxdt\lesssim C_{Carl}e^{-\frac{r_-\cdot r_+}{4C_{Carl}T_1}}\big(X+e^{\frac{2r_+^2}{C_{Carl}T_1}}Y\big),
\end{equation}
where 
\begin{align*}
X:=\iint\limits_{\mathcal A}e^{\frac{2|x|^2}{C_{Carl}T_{1}}}(T_{1}^{-1}|w|^2+|\nabla w|^2)\, dxdt,\qquad Y:=\int\limits_{r_-\leq |x|\leq r_+}|w(x,0)|^2\, dx.
\end{align*}
\end{proposition}
\begin{proposition}\label{pro:uniquecont}{(Quantitative unique continuation Carleman inequality)}
\\\cite[Proposition 22]{BMT}

Let $C_{0}\in[1,\infty)$, $T_{1}\in (0,\infty)$, $r>0$ and we define the space-time cylinder
\begin{equation*}
\mathcal C:=\{(x,t)\in\mathbb \R^3\times\mathbb \R\, :\ t\in [-T_{1},0],\ |x|\leq r\}.
\end{equation*}
Let $w:\ \mathcal C\rightarrow\mathbb \R^3$ be such that $w,\ \partial_tw,\ \nabla w$ and $\nabla^2w$ are continuous in space and time and such that $w$ satisfies the differential inequality
\begin{equation}\label{e.diffineqC}
|(\partial_{t}-\Delta) w(x,t)|\leq \frac{|w(x,t)|}{C_0 T_{1}}+\frac{|\nabla w (x,t)|}{(C_0 T_{1})^{\frac{1}{2}}}\quad\mbox{on}\ \mathcal C.
\end{equation}
Assume 
\begin{equation}\label{e.lowerr}
r^2\geq 16000T_{1}.
\end{equation}
Then, for all $0<\underline{s}\leq\overline s<\frac{T_{1}}{30000}$ one has the bound
\begin{equation}\label{e.conclcarltwo}
\int\limits_{-2\overline s}^{-\overline s}\int\limits_{|x|\leq \frac r2}(T_{1}^{-1}|w|^2+|\nabla w|^2)e^{\frac{|x|^2}{4t}}\, dxdt \lesssim e^{-\frac{r^2}{500\overline s}}X+(\overline s)^\frac{3}{2}\Big(\frac{3e\overline s}{\underline s}\Big)^{\frac{r^2}{200\overline s}}Y,
\end{equation}
where 
\begin{align*}
X:=\int\limits_{-T_{1}}^0\int\limits_{|x|\leq r}(T_{1}^{-1}|w|^2+|\nabla w|^2)\, dxdt,\qquad Y:=\int\limits_{|x|\leq r}|w(x,0)|^2(\underline s)^{-\frac{3}{2}}e^{-\frac{|x|^2}{4\underline s}}\, dx.
\end{align*}
\end{proposition}
Now we state and prove a key Proposition for proving Theorems \ref{thm:loweratblowup}-\ref{thm:lowernoref}, regarding the inductive application of quantitative Carleman inequalities. This will be used to iteratively transfer localized vorticity concentration to times to the future of a blow-up time, in a uniform way with respect to each iteration.
\begin{proposition}\label{pro:inductivebackuniqueness}{(Inductive backward uniqueness for the Navier-Stokes equations)}\\
Let $\lambda_{0}$ be as in \eqref{eq:lambdaolambda1def}.
There exists a positive universal constant $M_5>1$ such that the following statement holds true.\\
Let $v:\R^3\times (0,\infty)\rightarrow \R^3$ be a suitable weak Leray-Hopf solution and $M\geq M_{5}$.
Suppose there exists $\beta>0$ and $T^{**}>0$ such that for every $S\in [\frac{T^{**}}{1+\lambda_{0}}, T^{**}(1+\lambda_{0})]$ we have
\begin{equation}\label{eq:goodannulusinductive}
\begin{split}
&\sup_{\Big\{(x,t): |x|\geq M^{\beta}S^{\frac{1}{2}},\,\,t\in (0,S)\Big\}} t^{\frac{1+j}{2}} |\nabla^{j}v(x,t)|\lesssim M^{-9}\quad\textrm{for}\quad j=0,1\quad\textrm{and}\quad\\
&\sup_{\Big\{(x,t): |x|\geq M^{\beta}S^{\frac{1}{2}},\,\,t\in (0,S)\Big\}} t^{\frac{3}{2}} |\nabla \omega(x,t)|\lesssim M^{-9}.
\end{split}
\end{equation}
Furthermore, suppose that there exists 
\begin{equation}\label{eq:R1tildedef}
\tilde{R}_{1}\geq 2M^{\beta+1}(T^{**}(1+\lambda_{0}))^{\frac{1}{2}}
\end{equation} 
and $\gamma>0$ such that the vorticity $\omega$ satisfies
\begin{equation}\label{eq:ancestorinductive}
\int\limits_{\frac{T^{**}}{1+\lambda_{0}}}^{T^{**}} \int\limits_{\tilde{R}_{1}\leq |x|\leq \tilde{R}_{1}M^{\gamma}} |\omega(x,t)|^2 dxdt\geq (T^{**})^{\frac{1}{2}} e^{-\frac{3(\tilde{R}_{1})^2 M^{\gamma}}{T^{**}}}. 
\end{equation}
Then the above assumptions imply that
\begin{equation}\label{eq:descendentinductive}
\int\limits_{\frac{\tilde{R}_{1}}{20}\leq |x|\leq \frac{\tilde{R}_{1}}{20} M^{3\gamma+3}}
 |\omega (x,T)|^2 dx\geq T^{-\frac{1}{2}}e^{-(\frac{\tilde{R}_{1}}{20})^2 \frac{M^{(3\gamma+3)}}{T}}\quad \forall T\in [T^{**},T^{**}(1+\lambda_{0})].
\end{equation}
Moreover,
\begin{equation}\label{eq:descendentinductiveintegral}
\int\limits_{T^{**}}^{T^{**}(1+\lambda_0)} \int\limits_{\frac{\tilde{R}_{1}}{20}\leq |x|\leq \frac{\tilde{R}_{1}}{20}M^{3\gamma+3}} |\omega(x,t)|^2 dxdt\geq (T^{**}(1+\lambda_0))^{\frac{1}{2}} e^{-3(\frac{\tilde{R}_{1}}{20})^2\frac{M^{(3\gamma+3)}}{T^{**}(1+\lambda_0)}}. 
\end{equation}

\end{proposition}
\begin{proof}
The proof of Proposition \ref{pro:inductivebackuniqueness} follows \cite{BP21}, which is a physical space analogue of Tao's strategy \cite{Ta21}. By comparison with previous Navier-Stokes works using quantitative Carleman inequalities, the ratio between the original times where vorticity concentrates \eqref{eq:ancestorinductive} and the future times where concentration is transferred to  \eqref{eq:descendentinductive}-\eqref{eq:descendentinductiveintegral} is fixed and does not degenerate with respect to $M$. As explained in the Introduction, this avoids further exponential losses in Theorem \ref{thm:lowernoref}. As explained in the Introduction, it is also necessary to track the spatial regions where vorticity concentrates in the iteration procedure in Theorem \ref{thm:lowernoref}. Hence, our setting in Proposition \ref{pro:inductivebackuniqueness} involves tracking general parameters and avoiding $M$ needing to be large depending on those parameters. This necessitates careful bookkeeping.

Fix  $T\in [T^{**},T^{**}(1+\lambda_{0})]$ and let
\begin{equation}\label{eq:r-r+def}
r_{-}:=\frac{\tilde{R}_{1}}{10},\quad r_{+}:=10\tilde{R}_{1}M^{\gamma+1}\quad\textrm{and}\quad T_{1}:=\frac{T}{16000000}.
\end{equation}
Furthermore, from \eqref{eq:R1tildedef} we have
\begin{equation}\label{eq:R1T1}
\tilde{R}_{1}\geq 2MT^{\frac{1}{2}}\geq 2M(T_{1})^{\frac{1}{2}}.
\end{equation}
Since $T\in [T^{**},T^{**}(1+\lambda_{0})]$,  \eqref{eq:goodannulusinductive}-\eqref{eq:R1tildedef} implies that for $M$ sufficiently large (independent of $\gamma$) that for $j=0,1$
\begin{equation}\label{eq:vortbootstrap}
\begin{split}
&\sup_{(x,t)\in \{x:\frac{r_{-}}{2}<|x|<2r_{+}\}\times [-2T_1,0]} (T_{1})^{\frac{1+j}{2}}|\nabla^{j} v(x,t+T)|\lesssim M^{-9}\quad\textrm{and}\\
& \sup_{(x,t)\in \{x:\frac{r_{-}}{2}<|x|<2r_{+}\}\times [-2T_1,0]} (T_{1})^{\frac{3}{2}}|\nabla \omega(x,t+T)|\lesssim M^{-9}.
\end{split}
\end{equation}
Let $W_{0}:\{x:r_{-}<|x|<r_{+}\}\times [-T_1,0]\rightarrow \R^3$ be defined by
\begin{equation}\label{eq:W0def}
W_{0}(x,t):=\omega(x,T+t).
\end{equation}
From \eqref{eq:vortbootstrap}, the vorticity equation and classical parabolic bootstrap arguments\footnote{See arguments in \cite[Proposition 2.1]{NRS} or \cite[Lemma 6.1 (pg 134)]{gregory2014lecture}, for example.}, we get that 

\begin{equation}\label{eq:W0difinequality}
\begin{split}
&\partial_{t}W_{0},\,\,\nabla W_0\,\,\textrm{and}\,\,\nabla^2 W_0\quad\textrm{are}\,\,\textrm{space-time continuous}\,\,\textrm{in}
\\&\{x:r_{-}<|x|<r_{+}\}\times [-T_1,0]\quad\textrm{and}\\
&|(\partial_{t}-\Delta)W_{0}(x,t)|\lesssim \frac{|W_{0}(x,t)|}{M^9 T_1}+\frac{|\nabla W_{0}(x,t)|}{M^9 (T_{1})^{\frac{1}{2}}}\\
&\forall (x,t)\in \{x:r_{-}<|x|<r_{+}\}\times [-T_1,0].
\end{split}
\end{equation}
Now $$\frac{T_{1}}{4}=2\lambda_{0} T\geq 2\lambda_{0}T^{**}=T^{**}(1+\lambda_{0})-T^{**}(1-\lambda_{0})\geq T-\frac{T^{**}}{1+\lambda_{0}}.$$
Thus, $$ \frac{T^{**}}{1+\lambda_{0}}\geq T-\frac{T_{1}}{4}.$$
From this and \eqref{eq:ancestorinductive}
\begin{equation*}
\begin{split}
 &\int\limits_{-\frac{T_1}{4}}^{0}\int\limits_{10r_-\leq|x|\leq \frac{r_+}2}T_{1}^{-1}|W_{0}|^2\, dxdt=\int\limits_{T-\frac{T_1}{4}}^{T}\int\limits_{\tilde{R}_{1}\leq|x|\leq 5\tilde{R}_{1}M^{\gamma+1}}T_{1}^{-1}|\omega|^2\, dxdt\\
 &\geq \int\limits_{\frac{T^{**}}{1+\lambda_{0}}}^{T^{**}}\int\limits_{\tilde{R}_{1}\leq|x|\leq 5\tilde{R}_{1}M^{\gamma+1}}T_{1}^{-1}|\omega|^2\, dxdt\geq T_{1}^{-1}(T^{**})^{\frac{1}{2}} e^{-\frac{3(\tilde{R}_{1})^2 M^{\gamma}}{T^{**}}}.
\end{split}
\end{equation*}
Using this and $T\in [T^{**},T^{**}(1+\lambda_{0})]\subset [T^{**},2T^{**}]$, we infer that
\begin{equation}\label{eq:W0intlower}
\int\limits_{-\frac{T_1}{4}}^{0}\int\limits_{10r_-\leq|x|\leq \frac{r_+}2}T_{1}^{-1}|W_{0}|^2\, dxdt\gtrsim T^{-\frac{1}{2}}e^{-\frac{6(\tilde{R}_{1})^2 M^{\gamma}}{T}}.
\end{equation}
Utilizing \eqref{eq:r-r+def}-\eqref{eq:R1T1} and \eqref{eq:W0difinequality}, for $M$ sufficiently large we can apply quantitative backward uniqueness (Proposition \ref{pro:backuniqueness}) to $W_0$ on $\{x:r_-<|x|<r_+\}\times [-T_1,0]$ with $C_{carl}:=16000000$. This gives
\begin{equation}\label{e.conclcarloneW0}
T^{-\frac{1}{2}}e^{-\frac{6(\tilde{R}_{1})^2 M^{\gamma}}{T}}\lesssim e^{-\frac{(\tilde{R}_{1})^2 M^{\gamma+1}}{4T}}\big(X+e^{\frac{200(\tilde{R}_{1})^2 M^{2(\gamma+1)}}{T}}Y\big),
\end{equation}
where 
\begin{equation}\label{eq:Xdef}
\begin{split}
&X:=\int\limits_{-T_1}^{0}\int\limits_{r_-\leq |x|\leq r_+}e^{\frac{2|x|^2}{C_{Carl}T_{1}}}(T_{1}^{-1}|W_0|^2+|\nabla W_0|^2)\, dxdt\\
&=\int\limits_{T-\frac{T}{16000000}}^{T}\int\limits_{|x|\in [\frac{\tilde{R}_{1}}{10},10\tilde{R}_{1} M^{\gamma+1}]} e^{\frac{2|x|^2}{T}}(16000000 T^{-1}|\omega|^2+|\nabla\omega|^2) dxdt\quad\textrm{and}
\end{split}
\end{equation}
\begin{equation}\label{eq:Ydef}
Y:=\int\limits_{r_-\leq |x|\leq r_+}|W_0(x,0)|^2\, dx=\int\limits_{|x|\in [\frac{\tilde{R}_{1}}{10},10\tilde{R}_{1} M^{\gamma+1}]}|\omega(x,T)|^2 dx.
\end{equation}
Now for $M$ sufficiently large (independent of $\gamma$) we have that
$$\frac{(\tilde{R}_{1})^2 M^{\gamma+1}}{4T}-\frac{6(\tilde{R}_{1})^2 M^{\gamma}}{T}=\frac{(\tilde{R}_{1})^2}{T}\Big(\frac{M^{\frac{1}{2}}}{4}M^{\gamma+\frac{1}{2}}-6M^{\gamma}\Big)\geq \frac{(\tilde{R}_{1})^2 M^{\gamma+\frac{1}{2}}}{T}. $$
Thus \eqref{e.conclcarloneW0} implies that 
$$T^{-\frac{1}{2}}e^{\frac{(\tilde{R}_{1})^2 M^{\gamma+\frac{1}{2}}}{T}}\lesssim X+e^{\frac{200(\tilde{R}_{1})^2 M^{2(\gamma+1)}}{T}}Y. $$
This gives that either we must have
\begin{equation}\label{eq:casea}
\textbf{\textrm{(a)}}:\quad e^{\frac{200(\tilde{R}_{1})^2 M^{2(\gamma+1)}}{T}}\int\limits_{|x|\in [\frac{\tilde{R}_{1}}{10},10\tilde{R}_{1} M^{\gamma+1}]}|\omega(x,T)|^2 dx\gtrsim T^{-\frac{1}{2}}e^{\frac{(\tilde{R}_{1})^2 M^{\gamma+\frac{1}{2}}}{T}}\quad\textrm{or}
\end{equation}
\begin{equation}\label{eq:caseb}
\textbf{\textrm{(b)}}:\quad\int\limits_{T-\frac{T}{16000000}}^{T}\int\limits_{|x|\in [\frac{\tilde{R}_{1}}{10},10\tilde{R}_{1} M^{\gamma+1}]} e^{\frac{2|x|^2}{T}}( T^{-1}|\omega|^2+|\nabla\omega|^2) dxdt\gtrsim T^{-\frac{1}{2}}e^{\frac{(\tilde{R}_{1})^2 M^{\gamma+\frac{1}{2}}}{T}}.
\end{equation}
\textbf{Treating case (a)}\\
In case (a), for $M$ sufficiently large (independent of $\gamma$) we have that 
\begin{equation}\label{eq:casealowerbound}
\int\limits_{|x|\in [\frac{\tilde{R}_{1}}{10},10\tilde{R}_{1} M^{\gamma+1}]}|\omega(x,T)|^2 dx\gtrsim T^{-\frac{1}{2}}e^{-\frac{200(\tilde{R}_{1})^2 M^{2(\gamma+1)}}{T}}.
\end{equation}
Next using \eqref{eq:R1T1} we have that for $M$ sufficiently large (independent of $\gamma$) $$\Big(\frac{\tilde{R}_{1}}{20}\Big)^2 \frac{M^{(3\gamma+3)}}{T}-\frac{200(\tilde{R}_{1})^2 M^{2(\gamma+1)}}{T}\geq \frac{(\tilde{R}_{1})^2 M^{2\gamma}}{T}\Big(\frac{M^3}{40}-200M^2\Big)\geq M^{2}.$$
Thus, \eqref{eq:casealowerbound} implies that for $M$ sufficiently large (independent of $\gamma$) that
$$\int\limits_{\frac{\tilde{R}_{1}}{20}\leq |x|\leq \frac{\tilde{R}_{1}}{20} M^{3\gamma+3}}
 |\omega (x,T)|^2 dx\geq T^{-\frac{1}{2}}e^{-(\frac{\tilde{R}_{1}}{20})^2 \frac{M^{(3\gamma+3)}}{T}}$$
 as required.\\
\textbf{Treating case (b)}\\
Note that \eqref{eq:R1T1} implies that for $M$ sufficiently large (independent of $\gamma$) that
\begin{equation}\label{eq:casebnosim}
\int\limits_{T-\frac{T}{16000000}}^{T}\int\limits_{|x|\in [\frac{\tilde{R}_{1}}{10},10\tilde{R}_{1} M^{\gamma+1}]} e^{\frac{2|x|^2}{T}}( T^{-1}|\omega|^2+|\nabla\omega|^2) dxdt\geq T^{-\frac{1}{2}}e^{\frac{(\tilde{R}_{1})^2 M^{\gamma+\frac{1}{4}}}{T}}.
\end{equation}
Now, $[\frac{\tilde{R}_{1}}{10},10\tilde{R}_{1} M^{\gamma+1}] $ can be covered by $k_{cov}$ intervals 
$$\Big[\frac{\tilde{R}_{1}}{10},\frac{2\tilde{R}_{1}}{10}\Big]\cup\Big[\frac{2\tilde{R}_{1}}{10},\frac{4\tilde{R}_{1}}{10}\Big]\ldots\cup\Big[\frac{2^{k_{cov}-1}\tilde{R}_{1}}{10},\frac{2^{k_{cov}}\tilde{R}_{1}}{10}\Big] $$
with 
\begin{equation}\label{eq:kcovdef}
k_{cov}:=\left\lceil\frac{\log(100M^{\gamma+1})}{\log 2}\right\rceil.
\end{equation}
Using \eqref{eq:R1T1}, for $M$ sufficiently large (independent of $\gamma$) we have
\begin{equation*}
\begin{split}
k_{cov}&\leq \frac{\log(100M^{\gamma+1})}{\log 2}+1\leq  \frac{2\log(100M^{\gamma+1})}{\log 2}\leq \frac{200M^{4(\gamma+\frac{1}{4})}}{\log 2}\leq \frac{200 e^{4M^{\gamma+\tfrac{1}{4}}}}{\log 2}\\
&\leq e^{5M^{\gamma+\frac{1}{4}}}\leq e^{\frac{(\tilde{R}_{1})^2 M^{\gamma+\frac{1}{4}}}{2T}}.
\end{split}
\end{equation*}
From this, \eqref{eq:casebnosim} and the pigeonhole principle, we see that there exists 
\begin{equation}\label{eq:tildeR2def}
\tilde{R}_{2}\in \Big[\frac{2\tilde{R}_{1}}{10},20\tilde{R}_{1}M^{\gamma+1}\Big]
\end{equation}
such that for all $M$ sufficiently large (independent of $\gamma$)
$$
\int\limits_{T-\frac{T}{16000000}}^{T}\int\limits_{|x|\in [\frac{\tilde{R}_{2}}{2},\tilde{R}_{2}]} e^{\frac{2|x|^2}{T}}( T^{-1}|\omega|^2+|\nabla\omega|^2) dxdt\geq T^{-\frac{1}{2}}e^{\frac{(\tilde{R}_{1})^2 M^{\gamma+\frac{1}{4}}}{2T}}.
$$
Thus, 
\begin{equation}\label{eq:casebpigeonholetildeR2}
\int\limits_{T-\frac{T}{16000000}}^{T}\int\limits_{|x|\in [\frac{\tilde{R}_{2}}{2},\tilde{R}_{2}]} ( T^{-1}|\omega|^2+|\nabla\omega|^2) dxdt\geq T^{-\frac{1}{2}}e^{\frac{-2(\tilde{R}_{2})^2}{T}}.
\end{equation}
Using \eqref{eq:R1tildedef} and \eqref{eq:tildeR2def}, we see that for $M$ sufficiently large
\begin{equation}\label{eq:R2tildeinequality}
\frac{\tilde{R}_{2}}{2}\geq \frac{\tilde{R}_{1}}{10}\geq \frac{2M^{\beta+1} T^{\frac{1}{2}}}{10}\geq M^{\beta} T^{\frac{1}{2}}. 
\end{equation}
So from \eqref{eq:goodannulusinductive}, we have that for $j=0,1,$
\begin{equation}\label{eq:goodannulusinductiverecall}
\begin{split}
&\sup_{\Big\{(x,t): |x|\geq \frac{\tilde{R}_{2}}{2},\,\,t\in (0,T)\Big\}} t^{\frac{1+j}{2}} |\nabla^{j}v(x,t)|\lesssim M^{-9}\quad\textrm{and}\\
&\sup_{\Big\{(x,t): |x|\geq \frac{\tilde{R}_{2}}{2},\,\,t\in (0,T)\Big\}} t^{\frac{3}{2}} |\nabla \omega(x,t)|\lesssim M^{-9}.
\end{split}
\end{equation}
So for $s\in [0,\frac{T}{16000000}]$ we have that
$$
\int\limits_{T-s}^{T}\int\limits_{|x|\in [\frac{\tilde{R}_{2}}{2},\tilde{R}_{2}]} ( T^{-1}|\omega|^2+|\nabla\omega|^2) dxdt\lesssim M^{-18}T^{-\frac{1}{2}}\Big(\frac{(\tilde{R}_{2})^2}{T}\Big)^{\frac{3}{2}}\frac{s}{T}.
$$
Thus, for $M$ sufficiently large
\begin{equation}\label{eq:removenearT}
\begin{split}
&\int\limits_{T-Te^{-\frac{4(\tilde{R}_{2})^2}{T}}}^{T}\int\limits_{|x|\in [\frac{\tilde{R}_{2}}{2},\tilde{R}_{2}]} ( T^{-1}|\omega|^2+|\nabla\omega|^2) dxdt\leq M^{-17}T^{-\frac{1}{2}}\Big(\frac{(\tilde{R}_{2})^2}{T}\Big)^{\frac{3}{2}}e^{-\frac{4(\tilde{R}_{2})^2}{T}}\\
&\leq M^{-17}T^{-\frac{1}{2}}e^{\frac{3(\tilde{R}_{2})^2}{2T}-\frac{4(\tilde{R}_{2})^2}{T}}\leq M^{-17}T^{-\frac{1}{2}}e^{-\frac{2(\tilde{R}_{2})^2}{T}}.
\end{split}
\end{equation}
From \eqref{eq:R2tildeinequality}, we see that for $M$ sufficiently large that
\begin{equation}\label{eq:tildeR2exp}
{\frac{(\tilde{R}_{2})^2}{T}}\geq {M}.
\end{equation}
Combining this with \eqref{eq:casebpigeonholetildeR2} and \eqref{eq:removenearT} gives that for $M$ sufficiently large
\begin{equation}\label{eq:removenearTbyexp}
\int\limits_{T-\frac{T}{16000000}}^{T-Te^{-\frac{4(\tilde{R}_{2})^2}{T}}}\int\limits_{|x|\in [\frac{\tilde{R}_{2}}{2},\tilde{R}_{2}]} ( T^{-1}|\omega|^2+|\nabla\omega|^2) dxdt\geq T^{-\frac{1}{2}}e^{-\frac{3(\tilde{R}_{2})^2}{T}}.
\end{equation}
Now, $[T-\frac{T}{16000000},T-Te^{-\frac{4(\tilde{R}_{2})^2}{T}}] $ can be covered by $k'_{cov}$ intervals 
\begin{equation*}
\begin{split}
&[T-2Te^{-\frac{4(\tilde{R}_{2})^2}{T}},T-Te^{-\frac{4(\tilde{R}_{2})^2}{T}}]\cup [T-4Te^{-\frac{4(\tilde{R}_{2})^2}{T}},T-2Te^{-\frac{4(\tilde{R}_{2})^2}{T}}]\\
&\ldots [T-2^{k'_{cov}} Te^{-\frac{4(\tilde{R}_{2})^2}{T}},T-2^{k'_{cov}-1}Te^{-\frac{4(\tilde{R}_{2})^2}{T}}]
\end{split}
\end{equation*}
with 
\begin{equation}\label{eq:kcov'def}
k'_{cov}:=\left\lceil\frac{\frac{4(\tilde{R}_{2})^2}{T}-\log(16000000)}{\log 2}\right\rceil.
\end{equation}
Using \eqref{eq:tildeR2exp}, for $M$ sufficiently large (independent of $\gamma$) we have
\begin{equation*}
\begin{split}
k_{cov}'\leq e^{\frac{(\tilde{R}_{2})^2}{T}}.
\end{split}
\end{equation*}
From this, \eqref{eq:removenearTbyexp} and the pigeonhole principle, we see that there exists 
\begin{equation}\label{eq:t3def}
t_{3}\in \Big[Te^{-\frac{4(\tilde{R}_{2})^2}{T}}, \frac{T}{16000000}\Big]
\end{equation}
such that for all $M$ sufficiently large 
\begin{equation}\label{eq:casebpigeonholet3}
\int\limits_{T-2t_{3}}^{T-t_{3}}\int\limits_{|x|\in [\frac{\tilde{R}_{2}}{2},\tilde{R}_{2}]} ( T^{-1}|\omega|^2+|\nabla\omega|^2) dxdt\geq T^{-\frac{1}{2}}e^{\frac{-4(\tilde{R}_{2})^2}{T}}.
\end{equation}
Next, $\{x:\frac{\tilde{R}_{2}}{2}\leq |x|\leq \tilde{R}_{2}\}$ can be covered by $k_{cov}^{''}$ balls of radius $(t_{3})^{\frac{1}{2}}$ with centers in $\{x:\frac{\tilde{R}_{2}}{2}\leq |x|\leq \tilde{R}_{2}\}$. Using \eqref{eq:tildeR2exp} and \eqref{eq:t3def}, we have that for $M$ sufficiently large
\begin{equation}\label{eq:annuluscovernumber}
k_{cov}^{''}\leq \frac{C_{univ}(\tilde{R}_{2})^3}{(t_3)^{\frac{3}{2}}}\leq C_{univ}\Big(\frac{(\tilde{R}_{2})^2}{T}\Big)^{\frac{3}{2}} e^{\frac{6(\tilde{R}_{2})^2}{T}}\leq e^{\frac{8(\tilde{R}_{2})^2}{T}}.
\end{equation}
Using this, \eqref{eq:casebpigeonholet3} and the pigeonhole principle, we see that there exists
\begin{equation}\label{eq:x3def}
x_{(3)}\in \Big\{x:\frac{\tilde{R}_{2}}{2}\leq |x|\leq \tilde{R}_{2}\Big\}
\end{equation}
such that 
\begin{equation}\label{eq:pigeonholex3}
\int\limits_{T-2t_{3}}^{T-t_{3}}\int\limits_{B(x_{(3)}, (t_3)^{\frac{1}{2}})} ( T^{-1}|\omega|^2+|\nabla\omega|^2) dxdt\geq T^{-\frac{1}{2}}e^{\frac{-12(\tilde{R}_{2})^2}{T}}.
\end{equation}
Define 
\begin{equation}\label{eq:defS3r3}
S_{3}:=8000000t_{3},\quad \underline{s}_{3}=\overline{s}_{3}=t_{3}\quad\textrm{and}\quad r_{3}:=1000\tilde{R}_{2}\Big(\frac{t_3}{T}\Big)^{\frac{1}{2}}.
\end{equation}
Making use of \eqref{eq:t3def}, we have
\begin{equation}\label{eq:S3inequality}
S_{3}\leq \frac{T}{2}\quad\textrm{and}\quad \underline{s}_{3}=\overline{s}_{3}\leq \frac{S_{3}}{100000}.
\end{equation}
Using \eqref{eq:tildeR2exp}, \eqref{eq:t3def}, \eqref{eq:x3def} and \eqref{eq:defS3r3}, for $M$ sufficiently large we have
\begin{equation}\label{eq:r3x3inequality}
r_{3}^2\geq Mt_{3}\geq 16000 S_{3}\quad\textrm{and}\quad r_{3}\leq \frac{\tilde{R}_{2}}{4}\leq \frac{|x_{(3)}|}{2}.
\end{equation}
Using this, for $M$ sufficiently large we obtain that 
\begin{equation*}
\begin{split}
&B(x_{(3)}, (t_3)^{\frac{1}{2}})\subset B\Big(x_{(3)},\frac{r_3}{2}\Big)\subset B(x_{(3)},r_3)\subset B\Big(x_{(3)},\frac{|x_{(3)}|}{2}\Big)\\
&\subset \Big\{x:\frac{|x_{(3)}|}{2}<|x|<\frac{3|x_{(3)}|}{2}\Big\}.
\end{split}
\end{equation*}
This, together with \eqref{eq:x3def} and \eqref{eq:tildeR2def}, implies that for $M$ sufficiently large (independent of $\gamma$)
\begin{equation}\label{eq:r3r1tildeinequality}
\begin{split}
&B(x_{(3)}, r_3)\subset \Big\{x:\frac{\tilde{R}_{2}}{4}<|x|<\frac{3\tilde{R}_{2}}{2}\Big\}\subset \Big\{x:\frac{\tilde{R}_{1}}{20}<|x|<{30\tilde{R}_{1}}M^{\gamma+1}\Big\}\\
&\subset \Big\{x:\frac{\tilde{R}_{1}}{20}<|x|<\frac{\tilde{R}_{1}}{20}M^{3(\gamma+1)}\Big\}.
\end{split}
\end{equation}
Using similar reasoning, \eqref{eq:S3inequality} and \eqref{eq:R1tildedef}, we have for $M$ sufficiently large that
\begin{equation}\label{eq:r3goodannulus}
\begin{split}
&B(x_{(3)}, \tfrac{3r_3}{2})\times [T-\tfrac{3S_{3}}{2},T]\subset B(x_{(3)}, \tfrac{3|x_{(3)}|}{4})\times [\tfrac{T}{4},T]\subset \Big\{x:\frac{|x_{(3)}|}{4}<|x|\Big\}\times[\tfrac{T}{4},T]\\
&\subset \Big\{x:\frac{\tilde{R}_{1}}{40}<|x|\Big\}\times[\tfrac{T}{4},T]\subset \{x:|x|\geq M^{\beta} T^{\frac{1}{2}}\}\times [\tfrac{T}{4},T].
\end{split}
\end{equation}

Thus, \eqref{eq:goodannulusinductive} implies that for $j=0,1,$
\begin{equation}\label{eq:uboundsuniquecont}
\begin{split}
&\|\nabla^{j} v\|_{L^{\infty}(B(x_{(3)}, \tfrac{3r_3}{2})\times [T-\tfrac{3S_{3}}{2},T])}\lesssim T^{-\frac{1+j}{2}}M^{-9}\quad\textrm{and}\\
&\|\nabla \omega\|_{L^{\infty}(B(x_{(3)}, \tfrac{3r_3}{2})\times [T-\tfrac{3S_{3}}{2},T])}\lesssim T^{-\frac{3}{2}}M^{-9}.
\end{split}
\end{equation}
Define $W_{1}:\{x:\,|x|\leq r_3\}\times [-S_{3},0]\rightarrow\R^3$ by
\begin{equation}\label{eq:W1def}
W_{1}(x,t):=\omega(x+x_{(3)}, T+t).
\end{equation}
Taking into account  \eqref{eq:uboundsuniquecont}, the vorticity equation, classical parabolic bootstrap arguments and \eqref{eq:S3inequality}, we have 
 that
\begin{equation}\label{eq:W1difinequality}
\begin{split}
&\partial_{t}W_{1},\,\,\nabla W_1\,\,\textrm{and}\,\,\nabla^2 W_1\quad\textrm{are}\,\,\textrm{space-time continuous}\,\,\textrm{in}\,\,\{x:|x|\leq r_{3}\}\times [-S_3,0],\\
&\sup_{(x,t)\in \{x:|x|\leq r_{3}\}\times [-S_{3},0]} T^{\frac{2+j}{2}}|\nabla^{j} W_{1}(x,t)|\lesssim M^{-9}\quad\textrm{for}\quad j=0,1\quad\textrm{and}\\
&|(\partial_{t}-\Delta)W_{1}(x,t)|\leq \frac{|W_{1}(x,t)|}{ S_{3}}+\frac{|\nabla W_{1}(x,t)|}{ (S_{3})^{\frac{1}{2}}}\\
&\forall (x,t)\in \{x:|x|\leq r_{3}\}\times [-S_3,0]\quad (\textrm{ for}\,\,M\,\,\textrm{sufficiently large}).
\end{split} 
\end{equation}
This, together with \eqref{eq:S3inequality}-\eqref{eq:r3x3inequality}, allows us to apply quantitative unique continuation (Proposition \ref{pro:uniquecont}) to $W_{1}$ on $\{x:|x|\leq r_3\}\times [-S_{3},0]$ with $C_{0}=1$. Here, $r_3$ and $\underline{s}_{3}=\overline{s}_{3}$ are as in \eqref{eq:defS3r3}. This gives
\begin{equation}\label{eq.uniquecontCarl}
Z:=\int\limits_{-2t_{3}}^{-t_3}\int\limits_{|x|\leq \frac{r_{3}}{2}}(S_{3}^{-1}|W_{1}|^2+|\nabla W_{1}|^2)e^{\frac{|x|^2}{4t}}\, dxdt \lesssim e^{-\frac{r_3^2}{500 t_3}}X+e^{\frac{3r_3^2}{200t_3}}Y,
\end{equation}
where 
\begin{align*}
X:=\int\limits_{-S_{3}}^0\int\limits_{|x|\leq r_3}(S_{3}^{-1}|W_{1}|^2+|\nabla W_{1}|^2)\, dxdt\quad\textrm{and}\quad Y:=\int\limits_{|x|\leq r_3}|W_{1}(x,0)|^2 e^{-\frac{|x|^2}{4t_{3}}}\, dx.
\end{align*}
Using \eqref{eq:W1difinequality}, \eqref{eq:S3inequality} and \eqref{eq:t3def}, we get
\begin{equation*}
\begin{split}
&X\lesssim M^{-18}r_3^3(T^{-2}+S_{3}T^{-3})\lesssim M^{-18}r_3^3T^{-2}\\
&=M^{-18}\Big(\frac{r_3^2}{T}\Big)^{\frac{3}{2}} T^{-\frac{1}{2}}\leq M^{-18}\Big(\frac{r_3^2}{t_3}\Big)^{\frac{3}{2}} T^{-\frac{1}{2}}.
\end{split}
\end{equation*}
From this and \eqref{eq:defS3r3}, we see that 
\begin{equation}\label{eq:Xexpest}
e^{-\frac{r_3^2}{500 t_3}}X\lesssim e^{-\frac{r_3^2}{500 t_3}}M^{-18}\Big(\frac{r_3^2}{t_3}\Big)^{\frac{3}{2}} T^{-\frac{1}{2}}\lesssim e^{-\frac{r_3^2}{1000 t_3}}M^{-18}T^{-\frac{1}{2}}=e^{-\frac{1000(\tilde{R}_{2})^2}{T}}M^{-18}T^{-\frac{1}{2}}.
\end{equation}
Using a change of variables, \eqref{eq:defS3r3} and \eqref{eq:r3r1tildeinequality} gives
\begin{equation}\label{eq:Yexpest}
e^{\frac{3r_3^2}{200t_3}}Y\leq e^{\frac{15000(\tilde{R}_{2})^2}{T}}\int\limits_{\frac{\tilde{R}_{1}}{20}<|x|<\frac{\tilde{R}_{1}}{20}M^{3(\gamma+1)}} |\omega(y,T)|^2 dy
\end{equation}
Using  a change of variables, \eqref{eq:S3inequality}-\eqref{eq:r3x3inequality} and \eqref{eq:pigeonholex3} we have
\begin{equation}\label{eq:Zloweruniquecont}
\begin{split}
&Z\geq \int\limits_{T-2t_3}^{T-t_{3}}\int\limits_{B(x_{(3)}, \frac{r_3}{2})}(S_{3}^{-1}|\omega(y,s)|^2+|\nabla\omega(y,s)|^2) e^{-\frac{|y-x_{(3)}|^2}{4(T-s)}} dyds\\
& \geq \int\limits_{T-2t_3}^{T-t_{3}}\int\limits_{B(x_{(3)}, (t_3)^{\frac{1}{2}})}(S_{3}^{-1}|\omega(y,s)|^2+|\nabla\omega(y,s)|^2) e^{-\frac{|y-x_{(3)}|^2}{4(T-s)}} dyds\\
&\geq e^{-\frac{1}{4}} \int\limits_{T-2t_3}^{T-t_{3}}\int\limits_{B(x_{(3)}, (t_3)^{\frac{1}{2}})}T^{-1}|\omega(y,s)|^2+|\nabla\omega(y,s)|^2 dyds\geq e^{-\frac{1}{4}}T^{-\frac{1}{2}}e^{\frac{-12(\tilde{R}_{2})^2}{T}}
\end{split}
\end{equation}
Combining \eqref{eq.uniquecontCarl}-\eqref{eq:Zloweruniquecont} gives that for $M$ sufficiently large (independent of $\beta$ and $\gamma$) that
$$T^{-\frac{1}{2}}e^{\frac{-12(\tilde{R}_{2})^2}{T}}\lesssim e^{-\frac{1000(\tilde{R}_{2})^2}{T}}M^{-18}T^{-\frac{1}{2}}+e^{\frac{15000(\tilde{R}_{2})^2}{T}}\int\limits_{\frac{\tilde{R}_{1}}{20}<|x|<\frac{\tilde{R}_{1}}{20}M^{3(\gamma+1)}} |\omega(y,T)|^2 dy $$
From this, we see that for $M$ sufficiently large the second term in this inequality is negligible compared to the lower bound. This, along with \eqref{eq:R1T1} and \eqref{eq:tildeR2def}, gives that for $M$ sufficiently large (independent of $\gamma$) that
\begin{equation*}
\begin{split}
&\int\limits_{\frac{\tilde{R}_{1}}{20}<|x|<\frac{\tilde{R}_{1}}{20}M^{3(\gamma+1)}} |\omega(y,T)|^2 dy \gtrsim T^{-\frac{1}{2}}e^{-\frac{C_{univ} M^{2\gamma+2}}{T}(\frac{\tilde{R}_{1}}{20})^2}\\
&=T^{-\frac{1}{2}}e^{-(\frac{\tilde{R}_{1}}{20})^2 \frac{M^{(3\gamma+3)}}{T}}e^{\frac{M^{(3\gamma+3)}-C_{univ} M^{2\gamma+2}}{T}(\frac{\tilde{R}_{1}}{20})^2}\\\geq
& T^{-\frac{1}{2}}e^{-(\frac{\tilde{R}_{1}}{20})^2 \frac{M^{(3\gamma+3)}}{T}}e^{\frac{(M^{2\gamma})(M^3-C_{univ}M^2)}{T}(\frac{\tilde{R}_{1}}{20})^2}\geq T^{-\frac{1}{2}}e^{-(\frac{\tilde{R}_{1}}{20})^2 \frac{M^{(3\gamma+3)}}{T}} e^{M^2}. 
\end{split}
\end{equation*}
Thus, for $M$ sufficiently large (independent of $\gamma$) we get $$\int\limits_{\frac{\tilde{R}_{1}}{20}<|x|<\frac{\tilde{R}_{1}}{20}M^{3(\gamma+1)}} |\omega(y,T)|^2 dy\geq T^{-\frac{1}{2}}e^{-(\frac{\tilde{R}_{1}}{20})^2 \frac{M^{(3\gamma+3)}}{T}}.$$ So in all cases (case (a)-(b)) we arrive at \eqref{eq:descendentinductive}, as required.
Note that \eqref{eq:descendentinductiveintegral} can be obtained by directly integrating \eqref{eq:descendentinductive}.
\end{proof}
The next Proposition serves as the base step for iteratively applying the conclusions in the previous Proposition \ref{pro:inductivebackuniqueness}.
\begin{proposition}\label{pro:largescaleuniquecont}(Application of  quantitative unique continuation at large scales)
Let $p\in (3,\infty]$ and let $\lambda_{0}-\lambda_{1}$ be as in \eqref{eq:lambdaolambda1def}. There exists a positive universal constant $M_6>1$ such that the following statement holds true.\\
Let $v:\R^3\times (0,\infty)\rightarrow \R^3$ be a suitable weak Leray-Hopf solution, with initial data $v_{0}:\R^3\rightarrow\R^3$, that has a singular point at $t=T^*$.
Suppose that $v_{0}(x_1,x_2,x_3)$ is approximately axisymmetric and there exists $M\geq M_{6}$ such that
\begin{equation}\label{initialdataassumptionuniquecont}
\|v_{0}\|_{L^3(\R^3)}+\||x_3|^{1-\frac{3}{p}}|v_{0}(\cdot,x_3)|\|_{L^{p}(\R^3)}\leq M.
\end{equation}
Then  the above assumptions imply that 

\begin{equation}\label{eq:vortlowerboundR0largescaleprop}
\begin{split}
&\int\limits_{\frac{t}{1+\lambda_0}}^{t} \int\limits_{\frac{R_0}{2}\leq |x|\leq \frac{M^{130}R_0}{2}}|\omega(x,s)|^2 dxds\geq t^{\frac{1}{2}}e^{-\frac{3M^{130}}{t}(\frac{R_0}{2})^2}\\
&\forall t\in [M^{-12}T^*, T^*]\,\,\textrm{and}\,\,R_0\geq M^{709+\frac{30}{1-\frac{3}{p}}}t^{\frac{1}{2}}
\end{split}
\end{equation}
\end{proposition}
\begin{proof}
The strategy of the proof follows \cite{BP21}, which is a physical space analogue of Tao's strategy \cite{Ta21}. We provide details for the reader's convenience.

Fix 
\begin{equation}\label{eq:tx1fixed}
t\in [M^{-12}T^*,T^*]\,\,\textrm{and}\,\,|y|\geq M^{709+\frac{30}{1-\frac{3}{p}}}t^{\frac{1}{2}}.
\end{equation}
Let $I'\subseteq [(1-2\lambda_1)t, (1-\tfrac{3}{2}\lambda_{1})t]$ be as in Proposition \ref{pro:epochL3}. We may then write
\begin{equation}\label{eq:epochlengthdef}
I':=[s_{1}'-T_{1}', s_{1}']\quad\textrm{with}\quad T_{1}':=M^{-126}t.
\end{equation}
Fix
\begin{equation}\label{eq:defs'}
s'\in [s_{1}'-\tfrac{1}{4}T_{1}', s_{1}'].
\end{equation}
For $T_{1}'':=\tfrac{3}{4}T_1'$, we define
$W:\R^3\times [-T_{1}'',0]\rightarrow\R^3$ by
\begin{equation}\label{eq:Wdef}
W(x,s):=\omega(y+x, s'+s)\quad\textrm{with}\,\,s\in[-T_{1}'',0]. 
\end{equation}
From Proposition \ref{pro:epochL3} and the vorticity equation, we infer that for $M$ sufficiently large
\begin{equation}\label{eq:Wdifinequality}
\begin{split}
&W\in C^{\infty}(\R^3\times[-T_{1}'',0]),\\
&\sup_{(x,t)\in \R^3\times[-T_{1}'',0]} (T_{1}'')^{\frac{2+j}{2}}|\nabla^{j} W(x,t)|\lesssim M^{-2}\quad\textrm{for}\quad j=0,1\quad\textrm{and}\\
&|(\partial_{t}-\Delta)W(x,t)|\leq \frac{|W(x,t)|}{ T_{1}''}+\frac{|\nabla W(x,t)|}{ (T_{1}'')^{\frac{1}{2}}}\quad\forall (x,t)\in \R^3\times[-T_{1}'',0].
\end{split} 
\end{equation}
Define 
\begin{equation}\label{eq:r1timesdef}
r_{1}:=M|y|,\quad\overline{s}_{1}:=\frac{T_{1}''}{40000}\quad\textrm{and} \quad \underline{s}_{1}:=M^{-3}T_{1}''.
\end{equation}
Taking into account \eqref{eq:tx1fixed}, we see that for $M$ sufficiently large (independent of $p$) that
\begin{equation}\label{eq:ineqforCarleman}
r_1^2\geq M^{1420}t\geq 16000 T_{1}''\quad\textrm{and}\quad 0<\underline{s}_{1}<\overline{s}_{1}<\frac{T_{1}''}{30000}.
\end{equation}
From \eqref{eq:Wdifinequality} and \eqref{eq:ineqforCarleman}, we see that we can apply Proposition \ref{pro:uniquecont} to $W$ on $\{x:|x|\leq r_{1}\}\times [-T_{1}'',0]$ with $C_0=1$. Here, we take $\underline{s}=\underline{s}_{1}$ and $\overline{s}=\overline{s}_{1}$. This gives
\begin{equation}\label{eq.uniquecontCarllargescales}
Z:=\int\limits_{-2\overline{s}_{1}}^{-\overline{s}_{1}}\int\limits_{|x|\leq \frac{r_{1}}{2}}((T_1'')^{-1}|W|^2+|\nabla W|^2)e^{\frac{|x|^2}{4s}}\, dxds \lesssim e^{-\frac{r_1^2}{500 \overline{s}_{1}}}X+\Big(\frac{3e\overline{s}_{1}}{\underline{s}_{1}}\Big)^{\frac{3}{2}+\frac{r_1^2}{200\overline{s}_{1}}}Y,
\end{equation}
where 
\begin{align*}
X:=\int\limits_{-T_1''}^0\int\limits_{|x|\leq r_1}({T_1''})^{-1}|W|^2+|\nabla W|^2\, dxds,\qquad Y:=\int\limits_{|x|\leq r_1}|W(x,0)|^2 e^{-\frac{|x|^2}{4\underline{s}_{1}}}\, dx.
\end{align*}
Using a change of variables and \eqref{eq:tx1fixed}, we then get
\begin{equation}
\label{eq:Zlowerlargescale1}
\begin{split}
&Z\geq \int\limits_{s'-\frac{T_{1}''}{20000}}^{s'-\frac{T_{1}''}{40000}}\int\limits_{|x-y|\leq 2|y|}(T_1'')^{-1}|\omega(x,s)|^2e^{-\frac{|x-y|^2}{4(s'-s)}}\, dxds\\
&\geq (T_1'')^{-1}e^{-\frac{40000|y|^2}{T_{1}''}}\int\limits_{s'-\frac{T_{1}''}{20000}}^{s'-\frac{T_{1}''}{40000}}\int\limits_{|x|\leq M^{709+\frac{30}{1-\frac{3}{p}}} t^{\frac{1}{2}}}|\omega(x,s)|^2 dxds.
\end{split}
\end{equation}
Using that $I'\subseteq [(1-2\lambda_1)t, (1-\tfrac{3}{2}\lambda_{1})t] $ and $t\in [M^{-12}T^*, T^*]$, we have that
$[s'-\frac{T_{1}''}{20000},s'-\frac{T_{1}''}{40000}]\subset [M^{-13}T^*, (1-\lambda_{1})T^*].$ Combining these facts with Proposition \ref{pro:quantregionconc} and \eqref{eq:Zlowerlargescale1} gives
\begin{equation}
\label{eq:Zlowerlargescale2}
\begin{split}
&Z\geq (T_1'')^{-1}e^{-\frac{40000|y|^2}{T_{1}''}}\int\limits_{s'-\frac{T_{1}''}{20000}}^{s'-\frac{T_{1}''}{40000}}\int\limits_{|x|\leq M^{709+\frac{30}{1-\frac{3}{p}}} t^{\frac{1}{2}}}|\omega(x,s)|^2 dxds
\\
&\geq  (T_1'')^{-1}e^{-\frac{40000|y|^2}{T_{1}''}}\int\limits_{s'-\frac{T_{1}''}{20000}}^{s'-\frac{T_{1}''}{40000}}\int\limits_{|x|\leq M^{709+\frac{30}{1-\frac{3}{p}}} s^{\frac{1}{2}}}|\omega(x,s)|^2 dxds\\
&\geq (T_1'')^{-1}e^{-\frac{40000|y|^2}{T_{1}''}}\int\limits_{s'-\frac{T_{1}''}{20000}}^{s'-\frac{T_{1}''}{40000}}\frac{M^{-303}}{s^{\frac{1}{2}}} dxds.
\end{split}
\end{equation}
Using this, together with $ [s'-\frac{T_{1}''}{20000},s'-\frac{T_{1}''}{40000}]\subset I'\subseteq [(1-2\lambda_1)t, (1-\tfrac{3}{2}\lambda_{1})t]$, $t^{\frac{1}{2}}\leq M^{64}(T_1'')^{\frac{1}{2}}$ and \eqref{eq:tx1fixed}, yields that for $M$ sufficiently large
\begin{equation}\label{eq:Zlowerboundmain}
\begin{split}
&Z\geq (T_1'')^{-1}e^{-\frac{40000|y|^2}{T_{1}''}}\int\limits_{s'-\frac{T_{1}''}{20000}}^{s'-\frac{T_{1}''}{40000}}\frac{M^{-303}}{s^{\frac{1}{2}}} dxds\geq M^{-368}(T_1'')^{-\frac{1}{2}}e^{-\frac{40000|y|^2}{T_{1}''}}\\
&\geq (T_1'')^{-\frac{1}{2}}e^{-\frac{80000|y|^2}{T_{1}''}}.
\end{split}
\end{equation}
Using \eqref{eq:Wdifinequality} and the fact that
\begin{equation}\label{eq:x1large}
|y|\geq M(T_{1}'')^{\frac{1}{2}},
\end{equation} 
we infer that for $M$ sufficiently large (independent of $p$) that
\begin{equation}\label{eq:Xinequalitylargescale}
e^{-\frac{r_1^2}{500 \overline{s}_{1}}}X\lesssim e^{\frac{-80 M^2|y|^2}{T_1''}}(T_1'')^{-2} M^3|y|^3\leq(T_1'')^{-\frac{1}{2}}e^{\frac{-40 M^2|y|^2}{T_1''}}.
\end{equation}
This can be absorbed into the lower bound for $Z$ in \eqref{eq:Zlowerboundmain}, thus from \eqref{eq.uniquecontCarllargescales} we have
\begin{equation}\label{eq:uniquecontmainabsorb1}
\begin{split}
&(T_1'')^{-\frac{1}{2}}e^{-\frac{80000|y|^2}{T_{1}''}}\lesssim \Big(\frac{3eM^3}{40000}\Big)^{\frac{3}{2}+\frac{200M^2|y|^2}{T_{1}''}}(Y_{1}+Y_{2})\\
&\textrm{with}\quad Y_{1}:=\int\limits_{B(y,\frac{|y|}{2})} |\omega(x, s')|^2  e^{-\frac{M^3|x-y|^2}{4T_{1}''}}dx
\\&\textrm{and}\quad  Y_{2}:=\int\limits_{B(y, M|y|)\setminus B(y,\frac{|y|}{2})} |\omega(x, s')|^2  e^{-\frac{M^3|x-y|^2}{4T_{1}''}}dx.
\end{split}
\end{equation}
Using \eqref{eq:Wdifinequality} and \eqref{eq:x1large}, we get that for $M$ sufficiently large (independent of $p$) that
\begin{equation*}
\begin{split}
\Big(\frac{3eM^3}{40000}\Big)^{\frac{3}{2}+\frac{200M^2|y|^2}{T_{1}''}}Y_{2}&\lesssim e^{-\frac{M^3|y|^2}{16T_1''}}e^{\log\Big(\frac{eM^3}{20000}\Big)\Big(\frac{3}{2}+\frac{200M^2|y|^2}{T_{1}''}\Big)}M^3|y|^3 (T_{1}'')^{-2}\\
&\leq (T_1'')^{-\frac{1}{2}}e^{-\frac{M^3|y|^2}{32 T_1''}}.
\end{split}
\end{equation*}
This can be absorbed in the lower bound in \eqref{eq:uniquecontmainabsorb1}, giving that for $M$ sufficiently large (independent of $p$) that 
$$\int\limits_{B(y,\frac{|y|}{2})} |\omega(x,s')|^2 dx\geq  (T_{1}'')^{-\frac{1}{2}} e^{-\frac{M^3|y|^2}{T_1''}}. $$
Integrating this in $s'$ over $$ [s_{1}'-\tfrac{1}{4}T_{1}', s_{1}']\subset I'\subseteq [(1-2\lambda_1)t, (1-\tfrac{3}{2}\lambda_{1})t]$$ and noting that $T_{1}''=\tfrac{3}{4}T_{1}'=\frac{3}{4}M^{-126}t$ yields 
$$\int\limits_{(1-2\lambda_{1})t}^{t} \int\limits_{B(y,\frac{|y|}{2})}|\omega(x,s)|^2 dxds\gtrsim M^{-63}t^{\frac{1}{2}}e^{-\frac{4M^{129}|y|^2}{3t}}. $$
Taking into account \eqref{eq:lambdaolambda1def} and \eqref{eq:tx1fixed}, we obtain that for $M$ sufficiently large that

\begin{equation}\label{eq:vortlowerboundR0largescaleproof}
\begin{split}
&\int\limits_{\frac{t}{1+\lambda_0}}^{t} \int\limits_{\frac{R_0}{2}\leq |x|\leq \frac{M^{130}R_0}{2}}|\omega(x,s)|^2 dxds\geq t^{\frac{1}{2}}e^{-\frac{3M^{130}}{t}(\frac{R_0}{2})^2}\\
&\forall t\in [M^{-12}T^*, T^*]\,\,\textrm{and}\,\,R_0\geq M^{709+\frac{30}{1-\frac{3}{p}}}t^{\frac{1}{2}}
\end{split}
\end{equation}
as required.
\end{proof}

\section{Proof of Theorems \ref{thm:loweratblowup}-\ref{thm:lowernoref} }
\subsection{Proof of Theorem \ref{thm:loweratblowup}}
By Proposition \ref{pro:largescaleuniquecont} we have that
\begin{equation}\label{eq:vortlowerboundR0largescalerecap}
\begin{split}
&\int\limits_{\frac{T^*}{1+\lambda_0}}^{T^*} \int\limits_{\frac{R_0}{2}\leq |x|\leq \frac{M^{130}R_0}{2}}|\omega(x,s)|^2 dxds\geq (T^*)^{\frac{1}{2}}e^{-\frac{3M^{130}}{T^*}(\frac{R_0}{2})^2}\quad\forall R_0\geq M^{709+\frac{30}{1-\frac{3}{p}}}(T^*)^{\frac{1}{2}}.
\end{split}
\end{equation}
From Proposition \ref{prop:improvedannulus}, the improved annulus of regularity gives that for every $S>0$ we have that for $j=0,1$ that
\begin{equation}\label{eq:improvedannulusrecap}
\begin{split}
&\sup_{\Big\{(x,t): |x|\geq M^{\beta} S^{\frac{1}{2}},\,\,t\in (0,S)\Big\}} t^{\frac{1+j}{2}} |\nabla^{j}v(x,t)|\lesssim M^{-9}\quad\textrm{and}\quad\\
&\sup_{\Big\{(x,t): |x|\geq M^{\beta} S^{\frac{1}{2}},\,\,t\in (0,S)\Big\}} t^{\frac{3}{2}} |\nabla \omega(x,t)|\lesssim M^{-9}.
\end{split}
\end{equation}
with 
\begin{equation}\label{eq:betadef}
\beta:=700+\frac{30}{1-\frac{3}{p}}.
\end{equation}
Let
\begin{equation}\label{eq:tildeR1def}
\tilde{R}_{1}:=\frac{M^{709+\frac{30}{1-\frac{3}{p}}}(T^*)^{\frac{1}{2}}}{2}.
\end{equation}
Then for $M$ sufficiently large
\begin{equation}\label{eq:tildeR1inequality}
\tilde{R}_{1}=\frac{M^{9+\beta}({T^*}(1+\lambda_0))^{\frac{1}{2}}}{2(1+\lambda_0)^{\frac{1}{2}}}\geq 2M^{\beta+1}({T^*}(1+\lambda_0))^{\frac{1}{2}}.
\end{equation}
Then \eqref{eq:vortlowerboundR0largescalerecap}-\eqref{eq:improvedannulusrecap} and \eqref{eq:tildeR1inequality} allow us to apply Proposition \ref{pro:inductivebackuniqueness} with $\gamma=130$ giving
\begin{equation}\label{eq:vortpastblowupproof}
\int\limits_{\frac{\tilde{R}_{1}}{20}\leq |x|\leq \frac{\tilde{R}_{1}}{20}M^{393}}|\omega(x,T)|^2 dx\geq T^{-\frac{1}{2}}e^{-(\frac{\tilde{R}_{1}}{20})^2 \frac{M^{393}}{T}}\quad \forall T\in [T^{*},T^{*}(1+\lambda_{0})].
\end{equation}
Using this and arguing in the same way as in \cite{BP21}\footnote{The arguments in \cite{BP21} are based by arguments of Tao in \cite{Ta21}.} gives
\begin{equation}\label{eq:L2lowerproof}
\begin{split}
\int\limits_{M^{708+\frac{30}{1-\frac{3}{p}}}T^{\frac{1}{2}}\leq |x|\leq M^{1102+\frac{30}{1-\frac{3}{p}}}T^{\frac{1}{2}}} &|v(x,T)| dx\geq Te^{-M^{1813+\frac{60}{1-\frac{3}{p}}}}\\
&\forall T\in [T^{*},T^{*}(1+\lambda_{0})].
\end{split}
\end{equation}
Specifically \eqref{eq:L2lowerproof} is derived from \eqref{eq:vortpastblowupproof} by using the pigeonhole principle, the quantitative estimate \eqref{eq:improvedannulusrecap}, integration by parts and H\"{o}lder's inequality. This concludes the proof of Theorem \ref{thm:loweratblowup}.
\subsection{Proof of Theorem \ref{thm:lowernoref}}
Since $\|v_0\|_{H^{1}(\R^3)}\leq M$, from \cite{Le} any blow-up time $T^*$ satisfies
\begin{equation}\label{eq:blowuptimeestLeray}
C^{-1}M^{-4}\leq T^*\leq CM^4,
\end{equation}
where $C$ is a positive universal constant.
\subsubsection{Case 1: $T\in [M^{-5}, T^*(1+\lambda_0)]$}
The assumptions for the initial data in Theorem \ref{thm:lowernoref} allow us to deduce
that the conclusion \eqref{eq:L2lowerproof} of Theorem \ref{thm:loweratblowup} also holds here. Moreover the same arguments in the proof of Theorem \ref{thm:loweratblowup} apply for $T^*$ replaced by any $S\in [M^{-12}T^*, T^*]$ giving 
\begin{equation}\label{eq:L2lowerproofbeforeblowup1}
\begin{split}
\int\limits_{M^{708+\frac{30}{1-\frac{3}{p}}}T^{\frac{1}{2}}\leq |x|\leq M^{1102+\frac{30}{1-\frac{3}{p}}}T^{\frac{1}{2}}} &|v(x,T)| dx\geq Te^{-M^{1813+\frac{60}{1-\frac{3}{p}}}}\\
&\forall T\in [M^{-12}T^{*},T^{*}].
\end{split}
\end{equation}
Noting \eqref{eq:blowuptimeestLeray}, we have for $M$ sufficiently large that $M^{-12}T^*\leq M^{-5}$, so combining \eqref{eq:L2lowerproof} with \eqref{eq:L2lowerproofbeforeblowup1} yields
\begin{equation}\label{eq:L2lowerproofbeforeblowupmain}
\begin{split}
\int\limits_{M^{708+\frac{30}{1-\frac{3}{p}}}T^{\frac{1}{2}}\leq |x|\leq M^{1102+\frac{30}{1-\frac{3}{p}}}T^{\frac{1}{2}}} &|v(x,T)| dx\geq T e^{-M^{1813+\frac{60}{1-\frac{3}{p}}}}\\
&\forall T\in [M^{-5},T^{*}(1+\lambda_0)].
\end{split}
\end{equation}
This gives \eqref{eq:vlowerbeforeblowup}, which in turn implies \eqref{eq:vlowernoref} for $T\in [M^{-5}, T^*(1+\lambda_0)]$.
\subsubsection{Case 2: $T\in [T^*(1+\lambda_0), M^5]$}
Let $\beta$ be as in \eqref{eq:betadef} and define
\begin{equation}\label{eq:r0def}
r_{(0)}:= M^{\beta+12+\frac{10\log(20)}{\log(1+\lambda_0)}},\quad\gamma_{(0)}:=130\quad\textrm{and}\quad T_{(0)}:=T^*.
\end{equation}
Let
\begin{equation}\label{eq:kfinaldef}
k_{final}:=\Big\lfloor{\frac{\log(M^{10})}{\log(1+\lambda_0)}}\Big\rfloor
\end{equation}
and for $k=1\ldots k_{final}+1$ we define
\begin{equation}\label{eq:recursive}
r_{(k)}:= \frac{r_{(k-1)}}{20},\quad T_{(k)}:=(1+\lambda_0)T_{(k-1)}\quad\textrm{and}\quad \gamma_{(k)}=3\gamma_{(k-1)}+3.
\end{equation}
Then taking into account \eqref{eq:blowuptimeestLeray} we see that for $M$ sufficiently large
\begin{equation}\label{eq:Tkinal+1ineq}
M^{15}\geq 2M^{10}T^*\geq T_{(k_{final}+1)}\geq M^{10} T^*\geq M^5.
\end{equation}
Moreover, using this we also have
\begin{equation}\label{eq:Rkfinal+1ineq}
r_{(k_{final}+1)}\geq \frac{M^{\beta+12+\frac{10\log(20)}{\log(1+\lambda_0)}}}{20^{\frac{\log(M^{10})}{\log(1+\lambda_0)}+1}}=\frac{M^{\beta+12}}{20}\geq 2M^{\beta+1}(T_{(k_{final}+1)}(1+\lambda_0))^{\frac{1}{2}}.
\end{equation}
Furthermore,
\begin{equation}\label{eq:gammak+1inequality}
M^{\gamma_{(k_{final}+1)}}\leq M^{3^{k_{final}+2}\gamma_0}\leq M^{3^{3k_{final}}}\leq e^{M^{\frac{30\log(3)}{\log(1+\lambda_0)}+1}}.
\end{equation}
Recall from Proposition \ref{prop:improvedannulus}  that for all $S>0$ and
$j=0,1$
\begin{equation}\label{eq:improvedannulusrecap2.0}
\begin{split}
&\sup_{\Big\{(x,t): |x|\geq M^{\beta}S^{\frac{1}{2}},\,\,t\in (0,S)\Big\}} t^{\frac{1+j}{2}} |\nabla^{j}v(x,t)|\lesssim M^{-9}\quad\textrm{and}\\
&\sup_{\Big\{(x,t): |x|\geq M^{\beta}S^{\frac{1}{2}},\,\,t\in (0,S)\Big\}} t^{\frac{3}{2}} |\nabla \omega(x,t)|\lesssim M^{-9}.
\end{split}
\end{equation}
From \eqref{eq:blowuptimeestLeray} and \eqref{eq:r0def}, we have
\begin{equation}\label{eq:r0lowerT0}
2r_{(0)}\geq M^{\beta+9}(T_{(0)})^{\frac{1}{2}}= M^{709+\frac{30}{1-\frac{3}{p}}}(T^*)^{\frac{1}{2}}.
\end{equation}
So by Proposition \ref{pro:largescaleuniquecont} we have that
\begin{equation}\label{eq:vortlowerboundR0largescalerecap2}
\begin{split}
&\int\limits_{\frac{T_{(0)}}{1+\lambda_0}}^{T_{(0)}} \int\limits_{r_{(0)}\leq |x|\leq M^{\gamma_{(0)}}r_{(0)}}|\omega(x,s)|^2 dxds\geq (T_{(0)})^{\frac{1}{2}}e^{-\frac{3M^{\gamma_{(0)}}r_{(0)}^2}{T_{(0)}}}.
\end{split}
\end{equation}
Using this, \eqref{eq:r0def}-\eqref{eq:recursive} and \eqref{eq:improvedannulusrecap2.0}, we may apply Proposition \ref{pro:inductivebackuniqueness} inductively for $k=0,1\ldots k_{final}$. This yields that for $k=0,1\ldots k_{final}$
\begin{equation}\label{eq:inductiveapplied}
\int\limits_{r_{(k+1)}\leq |x|\leq r_{(k+1)}M^{\gamma_{(k+1)}}} |\omega(x,T)|^2 dx\geq T^{-\frac{1}{2}}e^{-\frac{(r_{(k+1)})^2 M^{\gamma_{(k+1)}}}{T}}\quad\forall T\in [T_{(k)}, T_{(k+1)}].
\end{equation}
From this, \eqref{eq:blowuptimeestLeray}, \eqref{eq:r0def} and \eqref{eq:recursive}-\eqref{eq:gammak+1inequality}, we infer that there exists $C(p)>1$ such that for $M$ sufficiently large (independent of $p$) we have
\begin{equation}\label{eq:vortinirregularregion}
\int\limits_{M^{\beta+11}\leq|x|\leq e^{M^{C(p)}}} |\omega(x,T)|^2 dx\geq e^{-e^{M^{C(p)}}}\quad\forall T\in [T^*, M^5].
\end{equation}
From \eqref{eq:improvedannulusrecap2.0} we have for $j=0,1$
\begin{equation}\label{eq:improvedannulusrecapM5}
\begin{split}
&\sup_{\Big\{(x,t): |x|\geq M^{\beta+3},\,\,t\in (0,M^5]\Big\}} t^{\frac{1+j}{2}} |\nabla^{j}v(x,t)|\lesssim M^{-9}\quad\textrm{and}\quad\\
&\sup_{\Big\{(x,t): |x|\geq M^{\beta+3},\,\,t\in (0,M^5]\Big\}} t^{\frac{3}{2}} |\nabla \omega(x,t)|\lesssim M^{-9}
\end{split}
\end{equation}
We can then infer the desired conclusion \eqref{eq:vlowernoref} (for $T\in [T^*(1+\lambda_0), M^5]$) from \eqref{eq:vortinirregularregion} by using \eqref{eq:improvedannulusrecapM5} and the arguments in \cite{Ta21} (subsequently used in \cite{BP21}), which is described at the end of Theorem \ref{thm:loweratblowup}. Combining with the Case 1, this concludes the proof.

\section{Blow-up from the right (Theorem \ref{thm:blowupright})}
For proving Theorem 3, we rely on the following proposition.
\begin{proposition}\label{pro:lowerboundinitialdatadecay}
Let $q\in [1,\infty]$ . There exists a positive universal constant $M_7>1$ such that the following statement holds true.\\
Let $v:\R^3\times (0,\infty)\rightarrow \R^3$ be a suitable weak Leray-Hopf solution, with initial data $v_{0}:\R^3\rightarrow\R^3$, that has a singular point at $(x,t)=(0,1)$.
Suppose that there exists $M\geq M_{7}$ such that
\begin{equation}\label{eq:initialdataassumptiondecay}
|v_{0}(x)|\leq \frac{M}{|x|}.
\end{equation}
Then  the above assumptions imply that 

\begin{equation}\label{eq:Lqlowerinitialdatadecay}
\begin{split}
&\Big(\int\limits_{M^{738}\leq|x|\leq M^{1133}}|v(x,T)|^q dx\Big)^{\frac{1}{q}}\geq e^{-M^{1874}}\quad \forall T\in [1, 1+\lambda_0].
\end{split}
\end{equation}
\end{proposition}
\begin{proof}
The proof is almost verbatim to Theorem \ref{thm:loweratblowup}. We provide details of the adjustments due to the initial data satisfying different assumptions in Proposition \ref{pro:lowerboundinitialdatadecay}, namely \eqref{eq:initialdataassumptiondecay}.

First note that the  bounds in Lemma \ref{energyboundweakL3} also hold in this setting. In particular, \eqref{eq:perturbationenergyest}-\eqref{L103} are obtained from \cite[Lemma 1]{TB24} using the embedding $L^{3}(\R^3)\hookrightarrow L^{3,\infty}(\R^3)$ and an appropriate rescaling so also apply for initial data bounded in $L^{3,\infty}(\R^3)$ as a result of \eqref{eq:initialdataassumptiondecay}. The bounds \eqref{eq:linearenergy} also hold true in our setting, following the same argument in Lemma \ref{energyboundweakL3} along with
$$\|e^{t\Delta}v_0\|_{L^{3,\infty}(\R^3)}\lesssim M $$ and H\"{o}lder's inequality for Lorentz spaces.

The conclusion of Proposition \ref{prop:improvedannulus} \eqref{eq:improvedannulus} also holds in our setting with $p=\infty$ there. This follows from the fact that \eqref{eq:initialdataassumptiondecay} implies that for $|x|\geq M^{730}$ and $M$ sufficiently large,  
$B(x, 3M^{600})\subseteq \{y: |y|\geq M^{729}\}$. Thus for $|x|\geq M^{730}$
$$\int\limits_{B(x, 3M^{600})}|v_0(y)|^3 dy\lesssim \frac{M^{1803}}{M^{2187}}\leq M^{-87}. $$
With this in hand, together with the fact that Lemma \ref{energyboundweakL3} also holds in our setting, verbatim arguments to Proposition \ref{prop:improvedannulus} give \eqref{eq:improvedannulus} with $p=\infty$.

The conclusion of Proposition \ref{pro:quantregionconc} \eqref{eq:vortconcnearblowup} also hold in our setting with $p=\infty$ and $T^*=1$. Indeed, as Lemma \ref{energyboundweakL3} also holds in our setting and as $(x,t)=(0,1)$ is a singular point, verbatim reasoning to Proposition \ref{pro:quantregionconc} gives 
$$\int\limits_{B(0, \tfrac{7}{2}M^{310}(1-t)^{\frac{1}{2}})} |\omega(x,t)|^2 dx\geq \frac{M^{-296}}{(1-t)^{\frac{1}{2}}}\quad\forall\quad t\in [1-2\lambda_1, 1-\lambda_1]\cap\Sigma. $$
Using this and \eqref{eq:distancefromTstar} with $T^*=1$ then gives \eqref{eq:vortconcnearblowup} with $p=\infty$ and $T^*=1$.

Finally, the conclusion of Proposition \ref{pro:epochL3} also holds in our setting. In particular, this used the bounds in Lemma \ref{energyboundweakL3}, which also hold in our setting.

Since Lemma \ref{energyboundweakL3}, Proposition \ref{prop:improvedannulus} (with $p=\infty$), Proposition \ref{pro:quantregionconc} (with $p=\infty$ and $T^*=1$) and Proposition \ref{pro:epochL3} hold, verbatim reasoning to the proof of Theorem \ref{thm:loweratblowup} gives \eqref{eq:L2lowerproof} with $p=\infty$ and $T^*=1$.
In particular, \begin{equation}\label{eq:L2lowerproofpinfty}
\begin{split}
\int\limits_{{M^{738}}T^{\frac{1}{2}}\leq |x|\leq M^{1132}T^{\frac{1}{2}}} &|v(x,T)| dx\geq Te^{-M^{1873}}\\
&\forall T\in [1,1+\lambda_{0}].
\end{split}
\end{equation}
H\"{o}lder's inequality then gives \eqref{eq:Lqlowerinitialdatadecay} for $M$ larger than a universal constant (independent of $q$).
\end{proof}
\subsection{Proof of Theorem \ref{thm:blowupright}}
\subsubsection{Case 1: Assumption \eqref{eq:logadjusted}, $L^{q}(\R^3)$ norm $q\in (3,\infty]$}

As in Theorem \ref{thm:blowupright}, let $0<c<\frac{1}{2}(1-\frac{3}{q})$. Recall assumption \eqref{eq:logadjusted}
\begin{equation}\label{eq:logadjustedrecall}
\begin{split}
&|v(x,t)|\leq \frac{M(t)}{|x|}\quad\textrm{for all}\quad T^*-\varepsilon_{0}(c, T^*)<t<T^*\\
& \textrm{with}\quad M(t):=\Big(\log\Big(\frac{1}{(T^*-t)^{c}}\Big)\Big)^{\frac{1}{1874}}.
\end{split}
\end{equation}
Take $\varepsilon_{0}(c, T^*)\in (0,T^*)$ such that $M(t)$ is well defined with $M(t)\geq M_{7}$ for all $T^*-\varepsilon_{0}(c, T^*)<t<T^*$, where $M_{7}$ is as in Proposition \ref{pro:lowerboundinitialdatadecay}.

For a fixed $t\in (T^*-\varepsilon_{0}(c, T^*),T^*)$ apply the rescaling 
\begin{equation}\label{eq:vrescale}
(v_{\lambda}, p_{\lambda})(x,s)=(\lambda v, \lambda^2 p)(\lambda x,\lambda^2 s +t)\quad\textrm{with}\quad \lambda=(T^*-t)^{\frac{1}{2}}.
\end{equation}
In particular, $v_{\lambda}:\R^3\times (0,\infty)\rightarrow\R^3$ is a suitable weak Leray-Hopf solution and possesses a singular point at $(x,s)=(0,1)$.
Also \eqref{eq:logadjustedrecall} implies that
\begin{equation}\label{eq:vlambdainitialdata}
|v_{\lambda}(x,0)|\leq \frac{M(t)}{|x|}\quad\textrm{with}\quad M(t)\geq M_{7}.
\end{equation}
This allows us to apply Proposition \ref{pro:lowerboundinitialdatadecay} to $v_{\lambda}$ with $M(t)$ in place of $M$, implying that if $M_{7}$ is a sufficiently large universal constant then
\begin{equation}\label{eq:Lqlowerinitialdatadecayrescaled}
\begin{split}
&\Big(\int\limits_{\R^3}|v(x,T^*+\lambda_0(T^*-t))|^q dx\Big)^{\frac{1}{q}} \\
&\geq \Big(\int\limits_{M(t)^{738}(T^*-t)^{\frac{1}{2}}\leq |x|\leq M(t)^{1133}(T^*-t)^{\frac{1}{2}}}|v(x,T^*+\lambda_0(T^*-t))|^q dx\Big)^{\frac{1}{q}}\\
&\geq \frac{e^{-(M(t))^{1874}}}{(T^*-t))^{\frac{1}{2}(1-\frac{3}{q})}}=\frac{(\lambda_0)^{\frac{1}{2}(1-\frac{3}{q})-c}}{(\lambda_0(T^*-t))^{\frac{1}{2}(1-\frac{3}{q})-c}}\quad\textrm{for}\quad t\in (T^*-\varepsilon_{0}(c, T^*),T^*).
\end{split}
\end{equation}

This then implies \eqref{eq:Lpfromright}, as desired.
\subsubsection{Case 2: Assumption \eqref{eq:doublelogadjusted}, $L^{3}(\R^3)$ norm}
Recall assumption \eqref{eq:doublelogadjusted}
\begin{equation}\label{eq:doublelogadjustedrecall}
\begin{split}
&|v(x,t)|\leq \frac{M(t)}{|x|}\quad\textrm{for all}\quad T^*-\varepsilon_{2}(T^*)\leq t<T^*\\
&\textrm{with}\quad M(t):=\Big(\log\log\Big(\frac{1}{(T^*-t)}\Big)\Big)^{\frac{1}{1876}}.
\end{split}
\end{equation}
Take $\varepsilon_{2}( T^*)\in (0,T^*)$ such that $M(t)$ is well defined with $M(t)\geq M_{7}$ for all $T^*-\varepsilon_{2}( T^*)\leq t<T^*$ where $M_{7}$ is as in Proposition \ref{pro:lowerboundinitialdatadecay}. From now on fix $t\in (T^*-\varepsilon_{2}( T^*),T^*)$. Note that for any $s\in [T^*-\varepsilon_{2}(T^*),t]$ we have
$$|v(x,s)|\leq \frac{M(t)}{|x|}. $$
Hence, using the similar arguments as in Case 1 (applying Proposition \ref{pro:lowerboundinitialdatadecay}) yields
\begin{equation}\label{eq:L3notsmall}
\begin{split}
\int\limits_{M(t)^{738}(T^*-s)^{\frac{1}{2}}\leq |x|\leq M(t)^{1133}(T^*-s)^{\frac{1}{2}}}&|v(x,T^*+\tau\lambda_0(T^*-s))|^3 dx \geq {e^{-(M(t))^{1875}}}\\
&\forall s\in [T^*-\varepsilon_{2}(T^*),t],\quad\forall\tau\in [0,1].
\end{split}
\end{equation}
As $\lambda_{0}(T^*-t)\leq\lambda_{0}(T^*-s)$, we infer from this that
\begin{equation}\label{eq:L3notsmallfixedtime}
\begin{split}
\int\limits_{M(t)^{738}(T^*-s)^{\frac{1}{2}}\leq |x|\leq M(t)^{1133}(T^*-s)^{\frac{1}{2}}}&|v(x,T^*+\lambda_0(T^*-t))|^3 dx \geq {e^{-(M(t))^{1875}}}\\
&\forall s\in [T^*-\varepsilon_{2}(T^*),t].
\end{split}
\end{equation}
Using \eqref{eq:L3notsmallfixedtime} we now employ a stacking of scales argument from \cite{Ta21}. First, we  further restrict $t\in (T^*-\varepsilon_{2}'(T^*), T^*)$, where $\varepsilon_{2}'(T^*)\in (0,\varepsilon_{2}(T^*))$ is such that $\log(M(t))$ is well defined with $$\frac{\log\Big(\frac{\varepsilon_{2}(T^*)}{T^*-t}\Big)}{790\log(M(t))}\geq 1.$$
For $
k=0,1\ldots$
define
\begin{equation}\label{eq:tkrecursive}
t_k=T^*-M(t)^{790k}(T^*-t).
\end{equation}
Note that
$$M(t)^{1133}(T^*-t_{k-1})^{\frac{1}{2}}=M(t)^{738}(T^*-t_k)$$
and for 
$$
k=0,1\ldots\Big\lfloor\frac{\log\Big(\frac{\varepsilon_{2}(T^*)}{T^*-t}\Big)}{790\log(M(t))}\Big\rfloor,\quad t_k=T^*-M(t)^{790k}(T^*-t)\geq T^*-\varepsilon_{2}(T^*).
$$
Hence, we can sum \eqref{eq:L3notsmallfixedtime} at each $t_{k}$ over 
$$
k=0,1\ldots\Big\lfloor\frac{\log\Big(\frac{\varepsilon_{2}(T^*)}{T^*-t}\Big)}{790\log(M(t))}\Big\rfloor.$$ Noting that the annuli in \eqref{eq:L3notsmallfixedtime} are disjoint at $s=t_{k}$ for distinct $k$, this yields
\begin{equation}\label{eq:L3summedfixedtime}
\begin{split}
\int\limits_{\R^3}&|v(x,T^*+\lambda_0(T^*-t))|^3 dx \geq \frac{\log\Big(\frac{\varepsilon_{2}(T^*)}{T^*-t}\Big){e^{-(M(t))^{1875}}}}{790\log(M(t))}\\
&\forall t\in (T^*-\varepsilon_{2}'(T^*), T^*).
\end{split}
\end{equation}
This and the expression for $M(t)$ in \eqref{eq:doublelogadjustedrecall}, together with an appropriate choice of $\varepsilon_{3}(T^*)$, readily gives 
\eqref{eq:L3fromright}.

\section*{Acknowledgments}
The author would like to thank Alexey Cheskidov and L\'{a}szl\'{o} Sz\'{e}kelyhidi for interesting discussions before the genesis of this paper. The author also thanks Mikhail Korobkov, Xiao Ren and Fei Wang for their valuable comments.

\section*{Funding}
The author is supported by the EPSRC New Investigator Award UKRI096 ‘Dynamics and
regularity criteria for nonlinear incompressible partial differential equations’.


\end{document}